\documentclass[11pt]{amsart}
\usepackage{t1enc}
\usepackage{amsthm,amssymb}
\usepackage{latexsym}
\usepackage{amsmath} 
\usepackage{graphics}
\usepackage{epsfig,multicol}
\usepackage{amsfonts}
\usepackage[latin1]{inputenc}
\usepackage[english]{babel}
\usepackage{ mathrsfs }
\usepackage{hyperref}
\hypersetup{
    colorlinks=true,
    linkcolor=blue,
    filecolor=magenta,      
    urlcolor=cyan,
}

\newtheorem{theorem}{Theorem}[subsection]
\newtheorem{proposition}{Proposition}[subsection]
\newtheorem{lemma}{Lemma}[subsection]
\newtheorem{definition}{Definition}[subsection]

\newtheorem{corollary}{Corollary}[subsection]
\newtheorem{remark}{Remark}[subsection]

\def\hpic #1 #2 {\mbox{$\begin{array}[c]{l} \epsfig{file=#1,height=#2}
\end{array}$}}
 
\def\vpic #1 #2 {\mbox{$\begin{array}[c]{l} \epsfig{file=#1,width=#2}
\end{array}$}}

\newcommand  {\rmn}\romannumeral

\newcommand{\g}{{\bf g}}

\begin{document}
\title{Irreducibility of the Wysiwyg representations  of Thompson's groups.}
\author{Vaughan F. R. Jones}
\thanks{}

\begin{abstract}
We prove irreducibility and mutual inequivalence for certain unitary representations of R. Thompson's groups F and T. 
\end{abstract}

\maketitle
\section{Introduction}
Let $F$ and $T$ be the Thompson groups as usual. In \cite{Jnogo} an action of $F$ was shown to arise from a \emph{functor} from the category $\mathcal F$  whose objects are natural numbers and whose morphisms are 
planar binary forests, to another category $\mathcal C$.
Forests decorated with cyclic permuations of their leaves give a category $\mathcal T$ for which functors from $\mathcal T$ give actions of $T$.

The representations studied in \cite{Jnogo} came from functors $\Phi$ to a trivalent tensor category (planar algebra) $\mathcal C$ in
the sense of \cite{MPS}, based on a specific "vacuum vector" $\Omega$ in the 1-box space of the tensor category.  A Thompson group 
element $g$ is represented by a pair of binary planar trees (see \cite{CFP}), drawn in the plane with one tree upside down on top of the other as below for an
element of $F$ that we will call D:

$\mbox{                       } \qquad D=  $\vpic {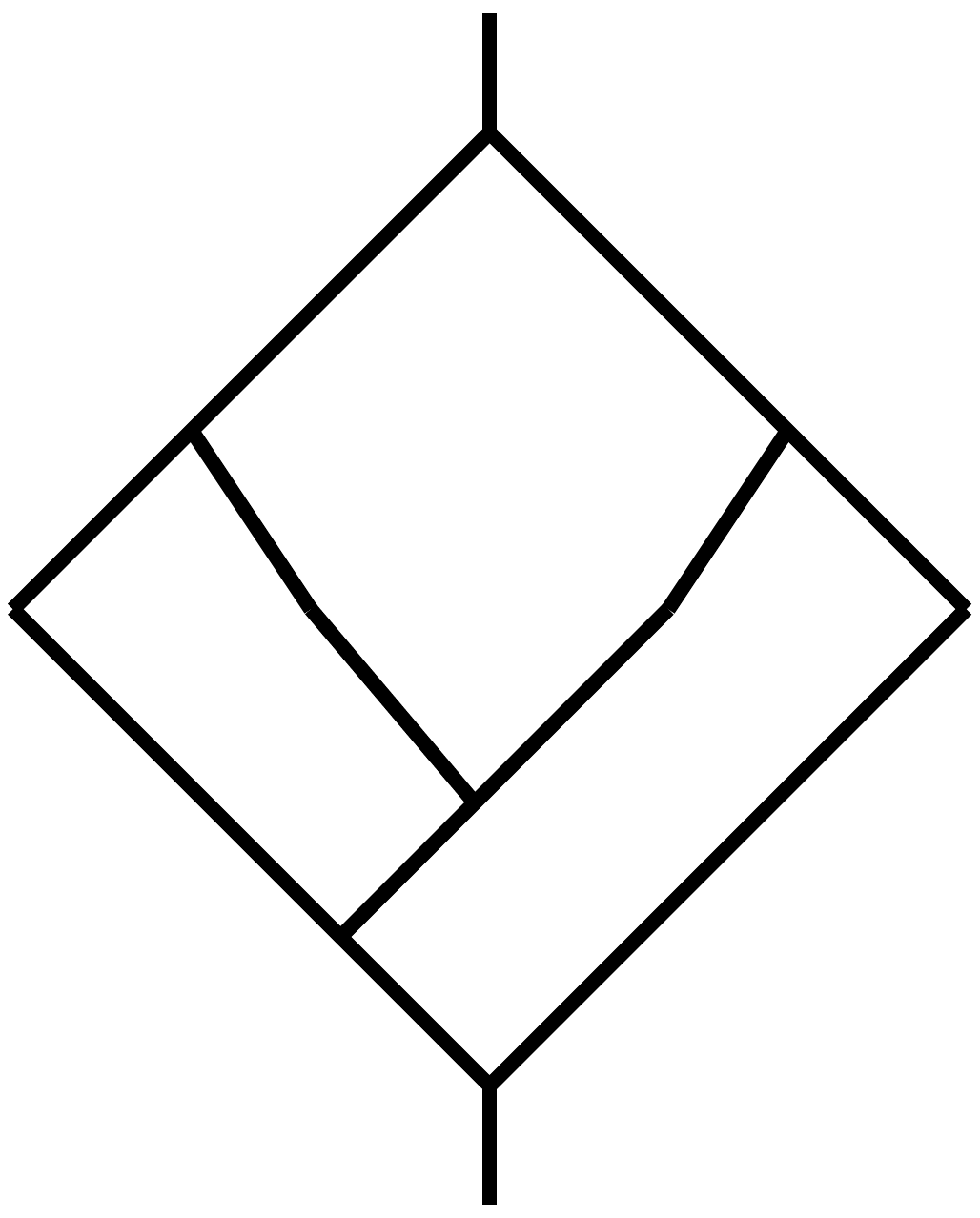} {1in}

The standard dyadic intervals defined by the leaves of the bottom tree are sent by $g$ (in the only affine way possible) to the corresponding intervals for the top tree. 

 If $\pi$ is the unitary representation defined by the (suitably normalised) trivalent vertex in $\mathcal C$, the coefficient 
 $$\langle \pi(g)\Omega, \Omega \rangle$$ is simply equal to the pair of trees of $g$ interpreted as a planar diagram (tangle) for
 $\mathcal C$ !! (Or more correctly the pair of trees as drawn is a multiple of a single vertical straight line and that multiple is
$\langle \pi(g)\Omega, \Omega \rangle$.)  For this reason we will call these representations, on the closure of the $F$-linear span of $\Omega$, the Wysiwyg (what you see is what you get) representations 
 of the Thompson groups. It is plausible that all the Wysiwyg representations are irreducible, indeed the entire family of 
 unitary representations defined in \cite{Jnogo} and \cite{Jthompson} could all be irreducible. Until this paper the only examples where irreducibility could
 be shown were when the representation was \emph{induced} from a subgroup. Then the problem becomes one of calculating
 the commensurator of that subgroup. This was done in \cite{gs} and \cite{NR}. 
 
 In this paper we will show that if we change the vacuum vector slightly then all the Wysiwyg representations are irreducible. The proof will not
 be difficult but comes from a remarkable piece of luck involving this new vacuum $\Psi$. To wit, $\Psi$ is fixed by the usual generator $A$
 of $F$ and any other fixed vector by $A$ is a multiple of $\Psi$. This immediately implies the representation is irreducible on the
 closed $F$-linear span of $\Psi$. (More information on the spectral measure of elements of $F$ can be found in \cite{AJ}.)
 
 Since these fixed vectors are canonical, any numerical data calculated from them are invariants of the representation.
 In this way we are able to show that these Wysiwyg representations are mutually inequivalent for different values of the parameter
 $d$ of the category $\mathcal C$. This was also unknown before, even for the induced representations.
 
 A more "elementary" way to get unitary representations is developed in \cite{BJ} where the tensor category is Hilbert spaces with
 direct sum as tensor product. Here too the coefficients of the vacuum vector are given simply by the pair of trees as above, but 
 interpreted as morphisms in the category. Little is known about irreducibility in this situation. 
 

\section{Definitions}

A binary planar forest is the isotopy class of a disjoint union of binary trees embedded
in $\mathbb R^2$ all of whose roots lie on $(\mathbb R,0)$ and all of whose leaves lie on $(\mathbb R,1)$. The isotopies are supported
in the strip $(\mathbb R,[0,1])$. Binary planar forests form a category
in the obvious way with objects being $\mathbb N$ whose elements are identified with isotopy classes of sets of points on a line and whose morphisms are the forests which can be composed by stacking a forest in 
$(\mathbb R,[0,1])$ on top of another, lining up the leaves of the one on the bottom with 
the roots of the other by isotopy then rescaling the $y$  axis to return
to a forest in  $(\mathbb R,[0,1])$. The structure is of course actually combinatorial but it is very useful to think of it in the way we have described.

We will call this category $\mathcal F$.

\begin{definition}\label{genforest} Fix $n\in \mathbb N$. For each $i=1,2,\cdots,n$ let $f_i$ be the planar
binary forest with $n$ roots and $n+1$ leaves consisting of straight
lines joining $(k,0)$ to $(k,1)$ for $1\leq k\leq i-1$ and $(k,0)$ to 
$(k+1,1)$ for $i+1\leq k\leq n$, and a single binary tree with root
$(i,0)$ and leaves $(i,1)$ and $(i+1,1)$ thus:

\hbox{       } \qquad\qquad \vpic {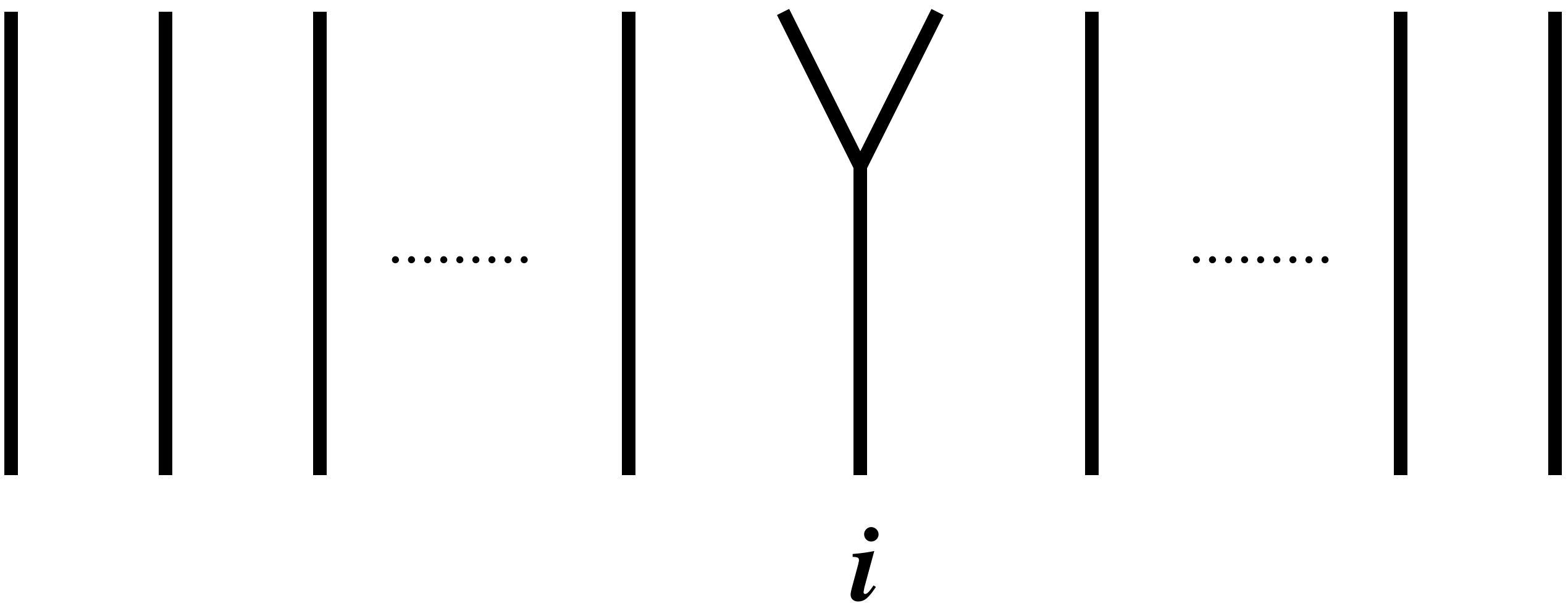} {2in}

\end{definition}

Note that any element of $\mathcal F$ is in an essentially unique way a 
composition of morphisms $f_i$, the only relation being $\Phi(f_j)\Phi(f_i)=\Phi(f_i)\Phi(f_{j-1}) \mbox{   for  } i<j-1$.
The set of 
morphisms from $1$ to $n$ in $\mathcal F$ is the set  of binary planar rooted trees $\mathfrak T$ and is a \emph{directed set} with $s\leq t$ iff there is and $f\in \mathcal F$ with $t=fs$.

Given a functor $\Phi:\mathcal  F\rightarrow \mathcal C$ to a category $\mathcal C$ whose objects are sets, 
we define the direct system $S_\Phi$ which associates to each $t  \in \mathfrak T$, $t:1\rightarrow n$,
 the set $\Phi(target(t))=\Phi(n)$. 
 For
each $s\leq t$ we need to give $\iota_s^t$. For this observe 
that there is an $f\in \mathcal F$ for which $t=fs$ so we
define $$\iota_s^t(\kappa) =\Phi(f)$$
 which is an element of $Mor_{\mathcal C}(\Phi(target(s)),\Phi(target(t)))$ as required. The $\iota_s^t$ trivially
 satisfy the axioms of a direct system.
 
  As a slight variation on this theme, given a functor $\Phi:\mathcal  F\rightarrow \mathcal C$ to \underline{any} category $\mathcal C$, 
and an object $\omega \in \mathcal C$, form the category $\mathcal C^\omega$ whose objects are the sets $Mor_{\mathcal C}(\omega,obj)$ for every 
object $obj$ in $\mathcal C$, and
whose morphisms are composition with those of $\mathcal C$. The definition of the functor $\Phi^\omega:\mathcal F\rightarrow \mathcal C^\omega$
is obvious.
  Thus the direct system $S_{\Phi^\omega}$ associates to each $t  \in \mathfrak T$, $t:1\rightarrow n$,
 the set $Mor_{\mathcal C}(\omega,\Phi(n))$. Given $s\leq t$ let  $f\in \mathcal F$ be such that $t=fs$.Then
for $\kappa \in Mor_{\mathcal C}(\omega,\Phi(target(s)))$,
 $$\iota_s^t(\kappa) =\Phi(f)\circ \kappa$$
 which is an element of $Mor_{\mathcal C}(\omega,\Phi(target(t)))$. 
 
 As in \cite{Jnogo} we consider the direct limit:
$$ \underset{\rightarrow} \lim S_\Phi=\{(t, x) \mbox{ with } t \in  \mathfrak T, x\in \Phi(target(t))\} / \sim$$ 
where $(t,x)\sim (s,y)$ iff there are $r\in \mathfrak T, z\in \Phi(target(z))$ with $t=fr, s=gr$ and $\Phi(f)(x)=z=\Phi(g)(y)$.
\vskip 5pt
\emph{We use $\displaystyle {t \over x}$ to denote the equivalence class of $(t,x)$ mod $\sim$.}
\vskip 5pt
The limit $ \underset{\rightarrow} \lim S_\Phi$ will inherit structure from
the category $\mathcal C$. For instance if the objects of $\mathcal C$ are Hilbert spaces and the morphisms are isometries then 
$ \underset{\rightarrow} \lim S_\Phi$ will be a pre-Hilbert space which may be completed to a Hilbert space which we will also call the direct
limit unless special care is required.

As was observed in \cite{Jnogo}, if we let 
$\Phi$ be the identity functor and choose $\omega$ to be the tree with one leaf,
then  the inductive limit consists of equivalence classes of pairs $\frac{t}{x}$ where $t\in \mathcal T$ and $x\in \Phi(target(t))=Mor(1,target(t))$.
But $Mor(1,target(t))$ is nothing but $s\in \mathcal T$ with $target(s)=target(t)$, i.e. trees with the same number of leaves as $t$.
Thus the inductive limit is nothing but the Thompson group $F$ with group law $${r\over s}{s\over t}={r\over t}.$$ 

 Moreover for any other functor $\Phi$,
$ \underset{\rightarrow} \lim S_\Phi$ carries a natural action of $F$ 
defined as follows:
$$\frac{s}{t}(\frac{t}{x})=\frac{s}{x}$$ where $s,t\in \mathfrak T$ with $target(s)=target(t)=n$ and $x\in \Phi(n)$.  A Thompon group element given as a pair of trees with $m$ leaves, and an element of
 $ \underset{\rightarrow} \lim S_\Phi$ given as a pair (tree with $n$ leaves,element of $\Phi(n)$), may not be immediately composable by the above forumula, but they can always be ``stabilised'' to be so within their equivalence classes. 
 
 The Thompson group action preserves the structure of 
 $ \underset{\rightarrow} \lim S_\Phi$ so for instance in the Hilbert space case the representations are unitary.
\section{The Wysiwyg representations.}

We are especially interested in applying the construction of the previous section when $\mathcal C$ is the category
of bifinite bimodules (correspondences in the sense of Connes) over a von Neumann algebra $M$ - \cite{C20}. 
Here we begin with a   Hilbert space $\mathcal H$  which is a  binormal bimodule for $M$ and form the category whose
objects are the relative tensor powers $\otimes^n_M \mathcal H$ which are also $M-M$ bimodules. The morphisms
of the category are $M-M$ bimodule maps. We will for simplicity restrict to certain such bimodule categories which
have been described simply diagramatically in \cite{MPS}.  To be precise:

\begin{definition} \label{trivalentcat} For $d\in \{4cos^2\pi/n -1|n=4,5,6,\cdots\}\cup [3,\infty)$ let $\mathcal C$ be the trivalent (tensor) category described in
\cite{MPS} with objects $1$ and $\otimes^nX$ for $n\geq 0$, with $\otimes^0 X=1$, for a privileged object $X$. 

By definition $\mathcal C_n$ is $Mor(1,\otimes^n X)$ which is the finite dimensional
Hilbert space of linear combinations of tangles with $n$ boundary points and an arbitrary number of  trivalent vertices, modulo the skein relations
 given in \cite{MPS}, and the kernel of the natural positive semidefinite Hermitian form.
We have $dim (\mathcal C_0)=1, dim( \mathcal C_1)=0, dim(\mathcal C_2=1)$ and $dim(\mathcal C_3)=1$.

\end{definition}

Here the object $1$ may be realised as $L^2(M)$ for a von Neumann algebra $M$ and $X$ would be some given $M-M$ correspondence.

The category $\mathcal C$ is by definition generated by an element in $\mathcal C_3$. Although \cite{MPS} does not address unitarity issues, 
 by realising $\mathcal C$ as
a cabled Temperley-Lieb  category as described  in \cite{MPS}, $\mathcal C$ has a $*$-structure with positive definite inner product.
The generator $Y$ of \cite{MPS} is then defined naturally from Temperley-Lieb and is invariant under the rotation. 
 $\mathcal C$ can also be realised as a category of correspondences over hyperfinite factors-\cite{HY}. (Indeed all of our considerations will apply to categories of $M-M$ bimodules  generated by $\mathcal H$ with suitable irreducibility and self-duality conditions, and an appropriate $M-M$ 
bilinear
$Y:\mathcal H\rightarrow \mathcal H\otimes_M\mathcal H$ with $Y^*Y=1$.)

\begin{definition} 
$Mor_{\mathcal C}(1, X\otimes X)$ is spanned by a single tangle consisting
 of a string joining the 2 boundary points. We will normalise this tangle by dividing 
it by $\sqrt d$ and call the resulting unit vector $\Psi$.

$Mor_{\mathcal C}(X,X \otimes X)$ is spanned by a single tangle consisting
 of the trivalent vertex. By our normalisation it is a unit vector.
We will call it  $\Omega$.
\end{definition}
The above choice of $Y$ gives a functor from the category of binary planar forests to $\mathcal C$ as described in \cite{Jnogo}.
The relation 2.1  of \cite{MPS} implies that the connecting maps $\iota_s^t$ are isometries. So
we obtain two categories $\mathcal C^1$ and $\mathcal C^X$ and direct systems 
$S_{\Phi^1}$ and $S_{\Phi^X}$. 
 Thompson's groups $F$ 
and $T$ act unitarily on the Hilbert spaces $\mathfrak H_Y^1$ and  $\mathfrak H_Y^X$ obtained as the (completions of the) direct limits 
$ \underset{\rightarrow} \lim S_{\Phi^1}$ and $ \underset{\rightarrow} \lim S_{\Phi^X}$ respectively.


Thus $\Phi^X$ is the functor which sends the object $n$ to the finite dimensional Hilbert space
 $\mathcal C_{n+1}=Mor_{\mathcal C}(X,\otimes^n X)$, and a morphism (forest)
$f:m\rightarrow n \in \mathcal F$ to the isometry $\Phi^X(f):\mathcal C_{m+1}\rightarrow \mathcal C_{n+1}$ obtained by
interpreting the forest $f$ as an element of $\mathcal C_{m+n}$ by replacing the vertices in $f$ with the element $Y\in \mathcal C$.
The element $\Phi^X(f)$ then has $m$ ingoing boundary points and $n$ outgoing ones so defines an isometry according to
the usual language of planar algebras -\cite{jo2}. We illustrate  below, for $x\in \Phi^X(3)$ and $f$ the forest with 3 roots and 8 leaves which is
visible in the diagram (thus $m=3, n=8$):

$\Phi^X(f)(x) \in Mor_{\mathcal C}(X,\otimes^n X)=  $ \vpic {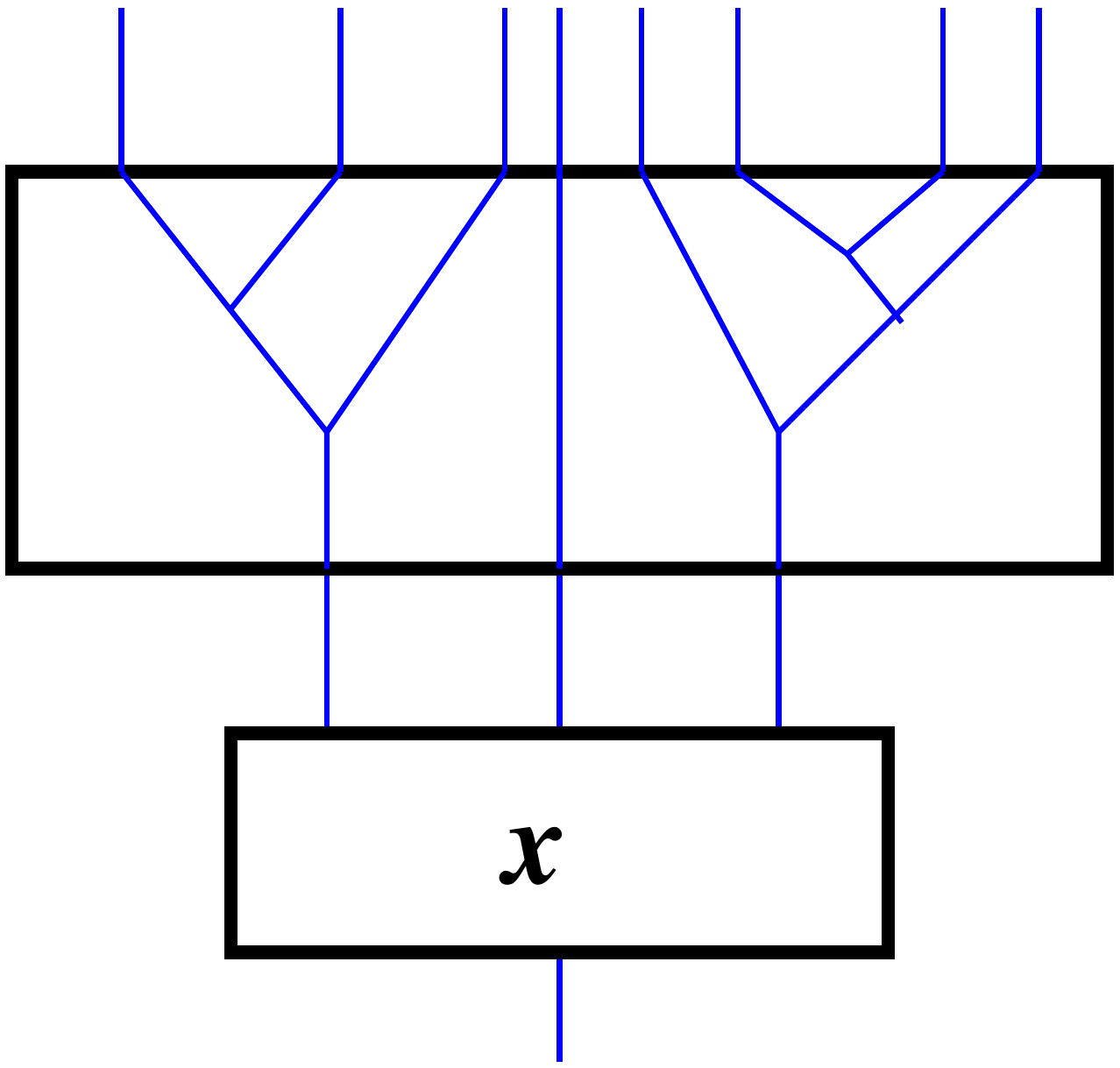} {1.5in}

The functor $\Phi^1$ is in some sense even simpler. We have

$$ \Phi^1(n)=
\begin{cases}
                                                  0 & \mbox { if } n=1 \\     
                                                 \mathcal C_n & \mbox{ if } n\geq 2
                                                \end{cases} $$
                                                
And as an illustration , for $x\in \Phi^1(3)$ and $f$ as above,

$\Phi^1(f)(x) \in Mor_{\mathcal C}(1,\otimes^n X)=  $ \vpic {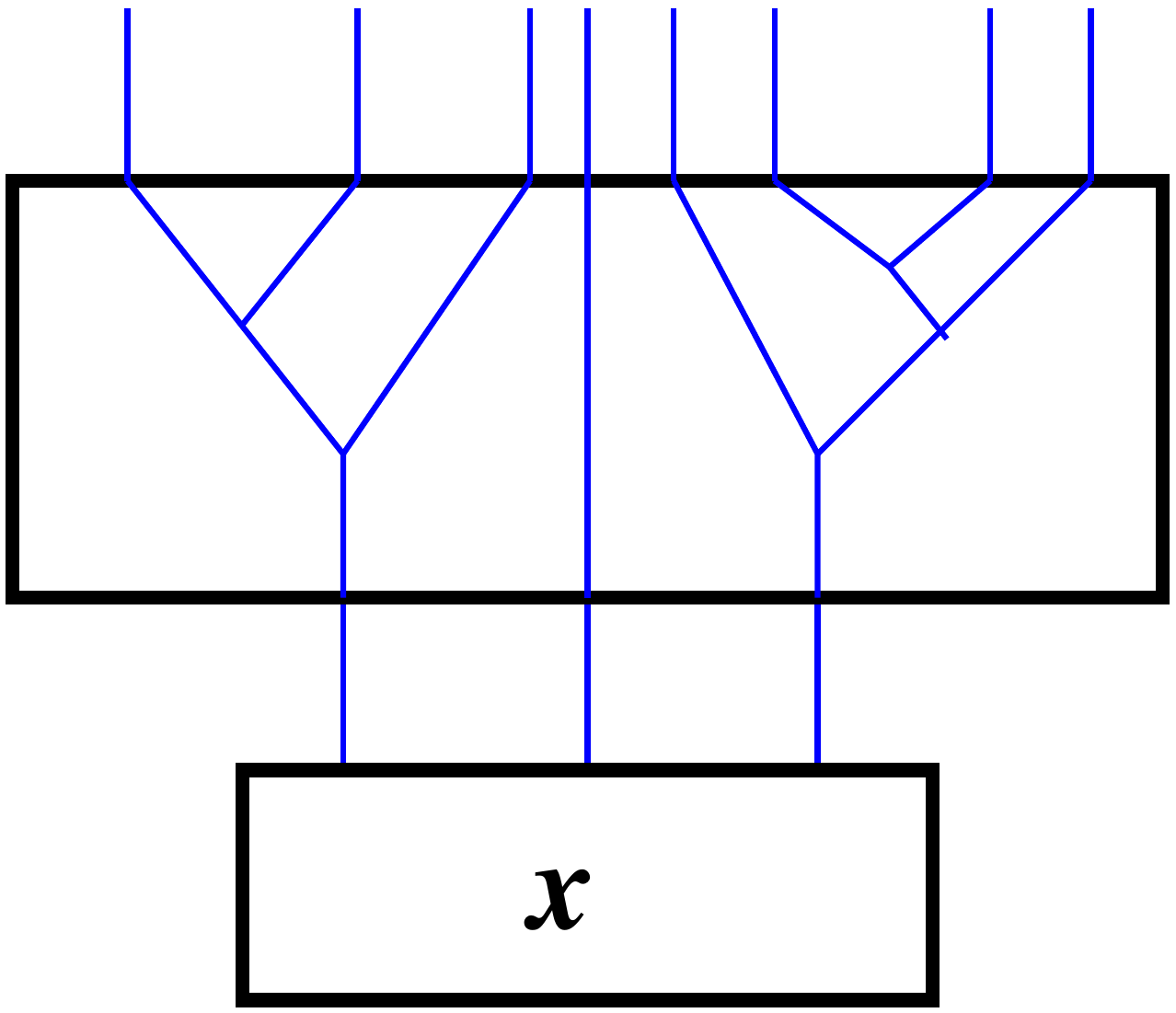} {1.5in}

When representing an element of Thompson's group $F$ as a pair of trees as in the introduction one can choose to either
include or not include the vertical edges at the top and bottom marking the roots. If $g\in F$ let $T_g$ be the version with
these vertices and $\mathbb T_g$ be the version without, e.g. for the element $D\in F$ of the introduction, the given
picture is $T_X$ and $\mathbb T_D$ is \vpic {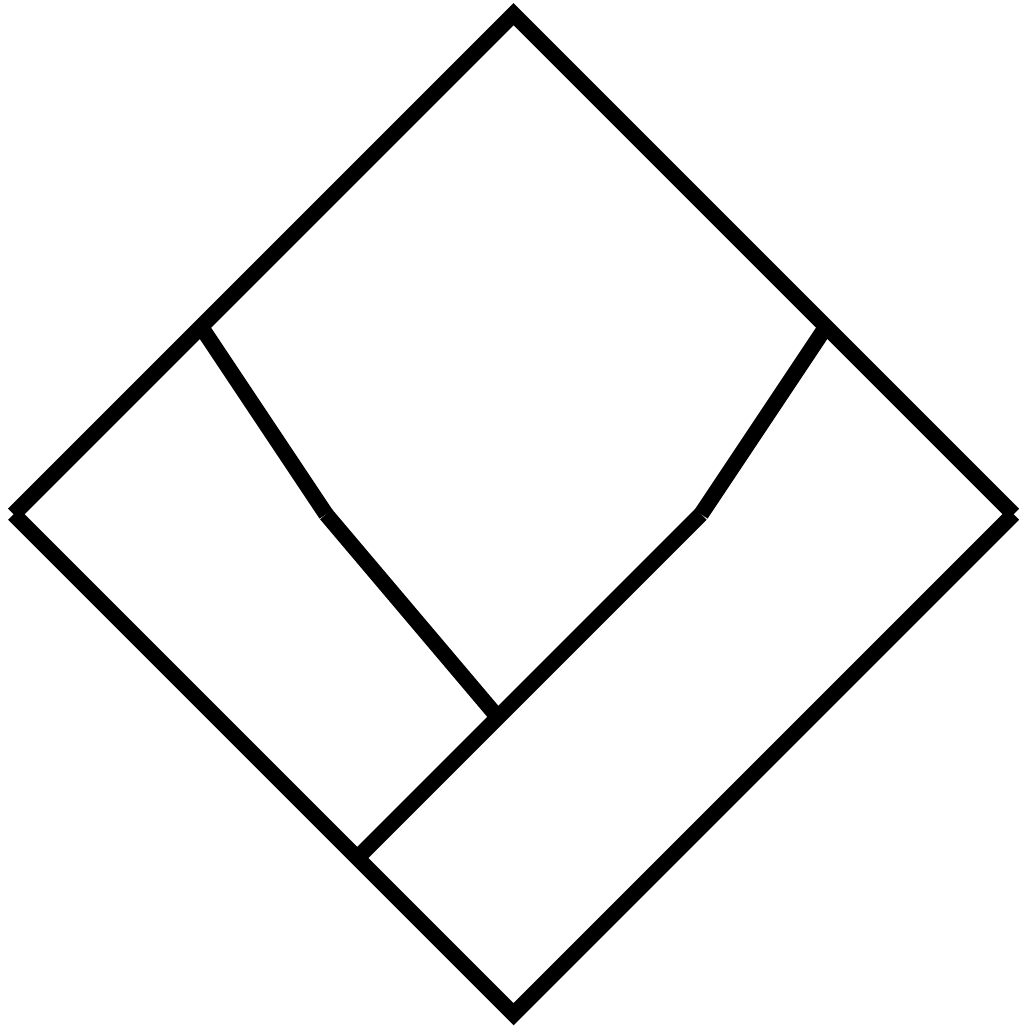} {1in}

\begin{proposition} Let $g\in F$ be given. Then the ``vacuum expectation value'', or ``vacuum coefficient'' 
$\langle \pi(g)\Omega,\Omega\rangle$ is ${1\over d}$ times the scalar obtained by tying the top of $T_g$ to the bottom and evaluating in 
$\mathcal C$ by replacing all vertices of the trees by $Y$'s. 

And $\langle \pi(g)\Psi,\Psi\rangle$ is ${1\over d}$ times the scalar obtained by replacing all the vertices of $\mathbb T_g$ by $Y$ and
evaluating in $\mathcal C$.
\end{proposition}
\begin{proof} This is an exercise which will be clear if we illustrate using the element $D\in F$ and the representation
coming from $\Phi^1$. 
\vskip 5pt
If $s$ is the tree \vpic {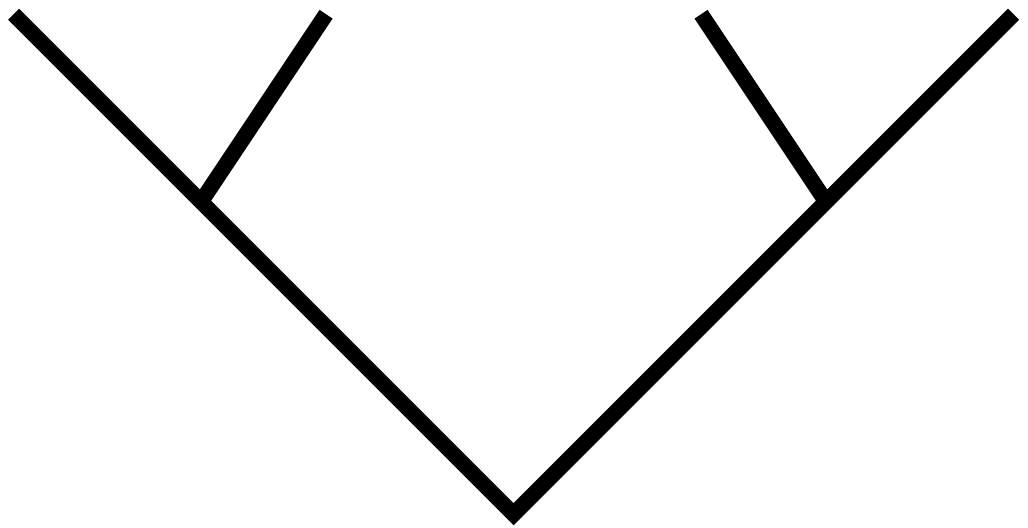} {0.8in} and $t$ is the tree \vpic {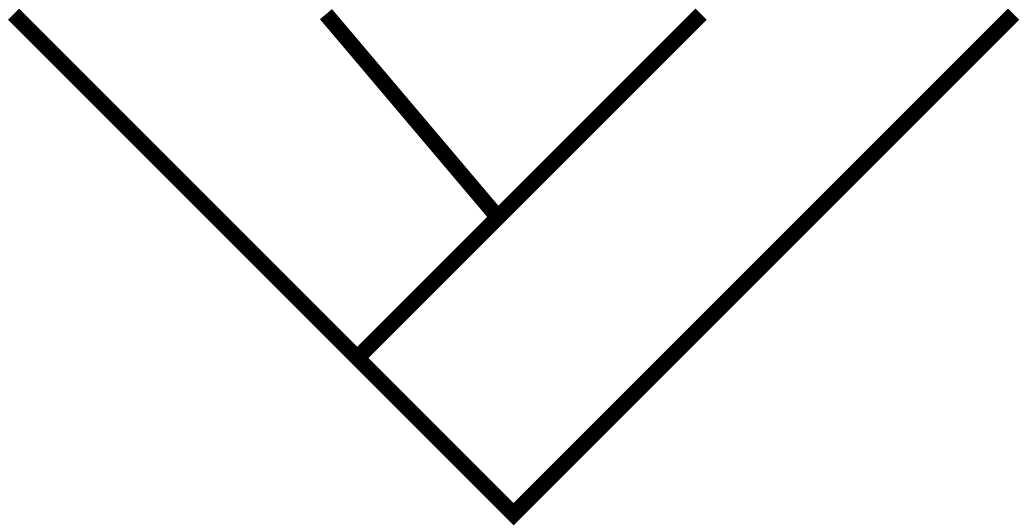} {0.8in} then $D={s\over t}$.  
As a vector in the direct limit, $\displaystyle \sqrt d \Psi={t\over \Phi^1(t)}$ so that $\displaystyle \sqrt d\pi(D)(\Psi)={s\over \Phi^1(t)}$.
But also $\sqrt d\displaystyle \Psi={s\over \Phi^1(s)}$ so by the definition of the inner product in the direct limit Hilbert space,
$d\langle \pi(g)\Psi,\Psi\rangle$ is the inner product in $\Phi^1(4)$ of $\Phi^1(t)$ with $\Phi^1(s)$ which by the 
definition of $\Phi$ is just \vpic {X1} {0.5in} viewed as an element of $\mathcal C_0=\mathbb C$.

The case of $\Phi^X$ is proved in the same way.

\end{proof}

\begin{definition} We call the representations $\pi_Y$ of $F$ coming from $\Phi^1$ and $\Phi^X$ as above, on the closed $F$-linear
span of  $\Psi$ and $\Omega$ respectively, the \emph{ Wysiwyg } representations of $F$.
\end{definition}
\begin{definition} \label{proj}
More generally one can let $p$ be any minimal projection in one of the algebras $\mathcal C_{2n}$ and consider the  direct system  $$t \mapsto \begin{cases} \mathcal C_m & \mbox{if $t$ is a tree with } m 
\mbox{ leaves and  } m\geq 2n\\
0 & \mbox{ if $t$ has less than $2m$ leaves. } 
\end{cases}$$Then $F$ acts on the direct limit as before and we let $\pi_{Y,p}$ be the unitary representation of $F$ on the 
Hilbert space completion, so that $\pi_{Y,\phi}$ contains $\pi_Y$ where $\phi$ is the identity of $\mathcal C_0$.
\end{definition}

\section{The main theorem}
Let $\mathcal C$  and $Y$ be as in the last section. Let $\pi_Y$ be the Wysiwyg representation of $F$ defined above with vacuum $\Psi$.

\begin{theorem} The unitary representation $\pi_Y$ is irreducible.\end{theorem}

Our proof will rely heavily on an algorithm for calculating the product of two elements of $F$ which is to be found in \cite{belk} and \cite{gus}.
Let us describe it:

Given two pairs $P$ and $Q$ of (binary planar rooted) trees represented diagrammatically as in the introduction, place 
$P$ on top of $Q$ in the plane and join them as below :

  \mbox{   \qquad} \hspace {1.1in}  \vpic {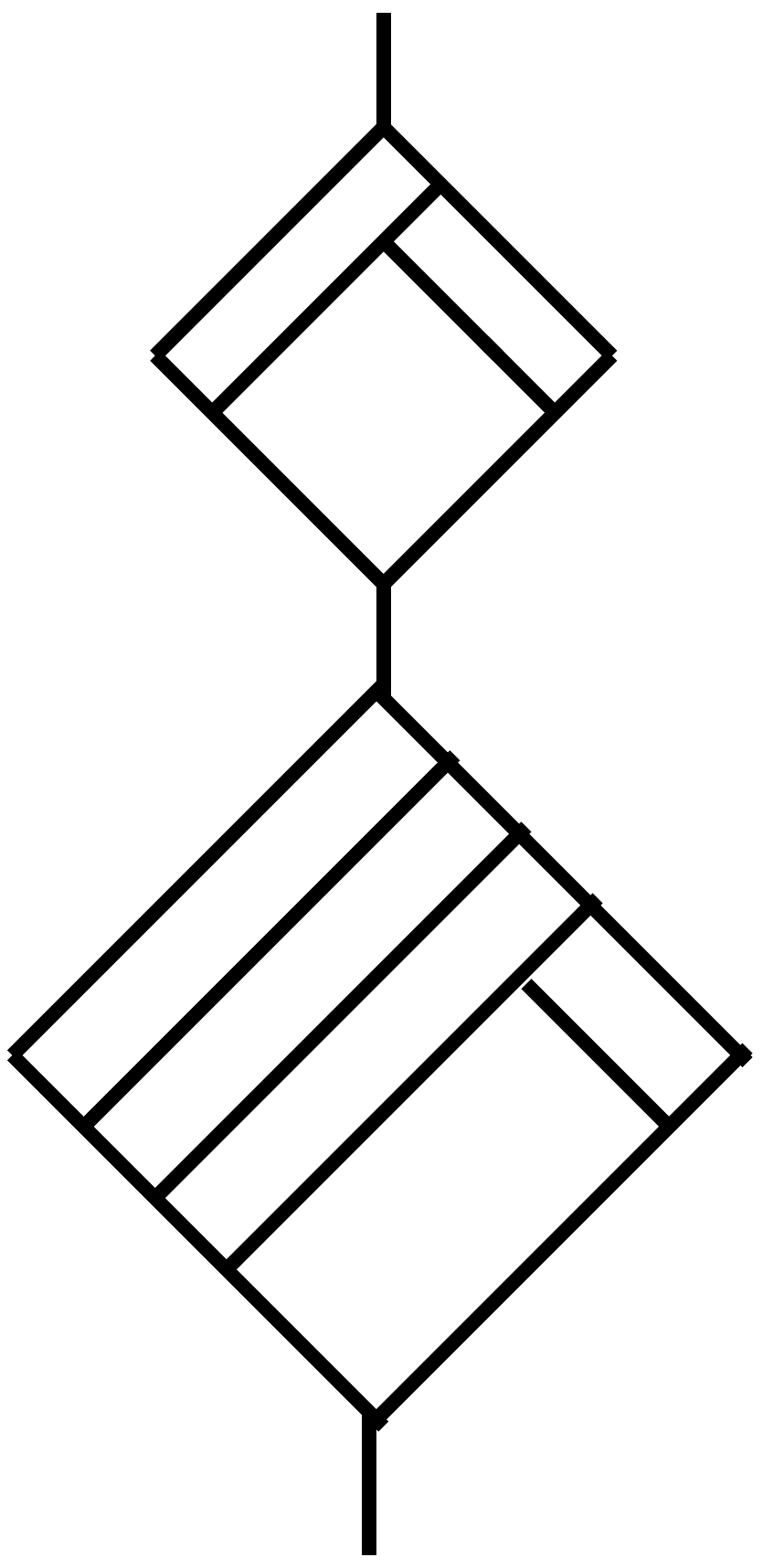} {1in}
  
  Then apply the diagrammatic rules \vpic {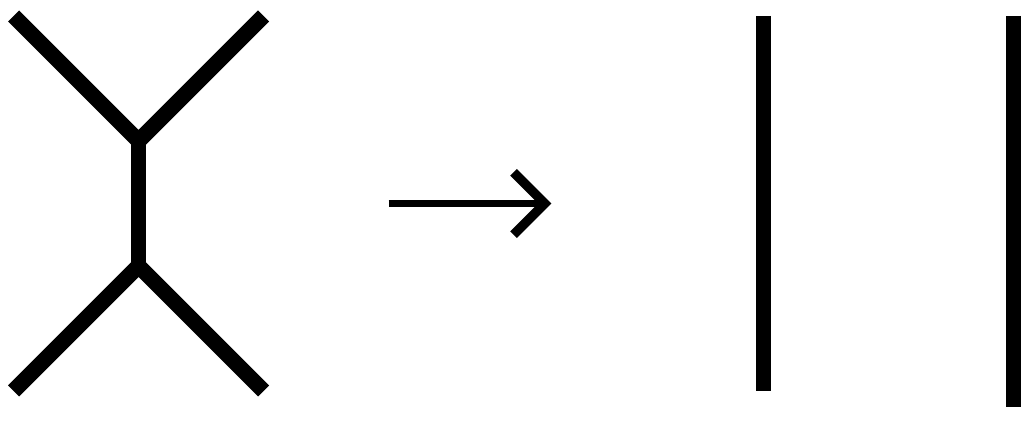} {0.5in} and \vpic {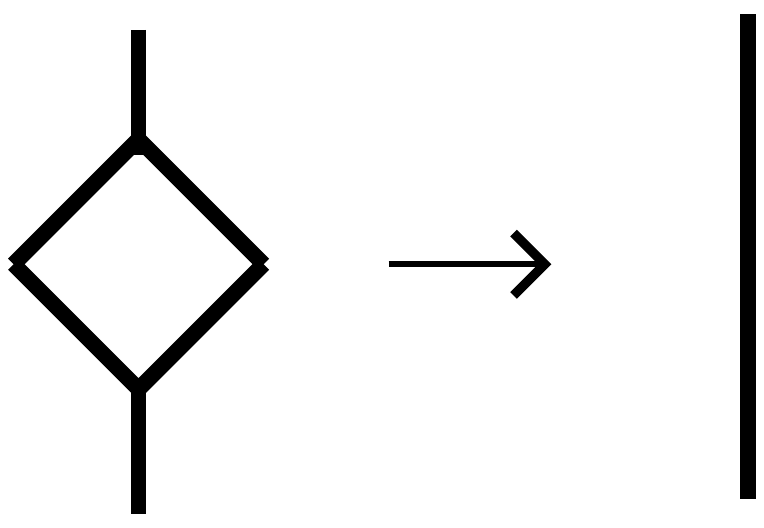} {0.5in} (and some isotopy) until they cannot be applied any more:
  
  \vpic {belk1} {0.8 in} $\rightarrow$ \vpic {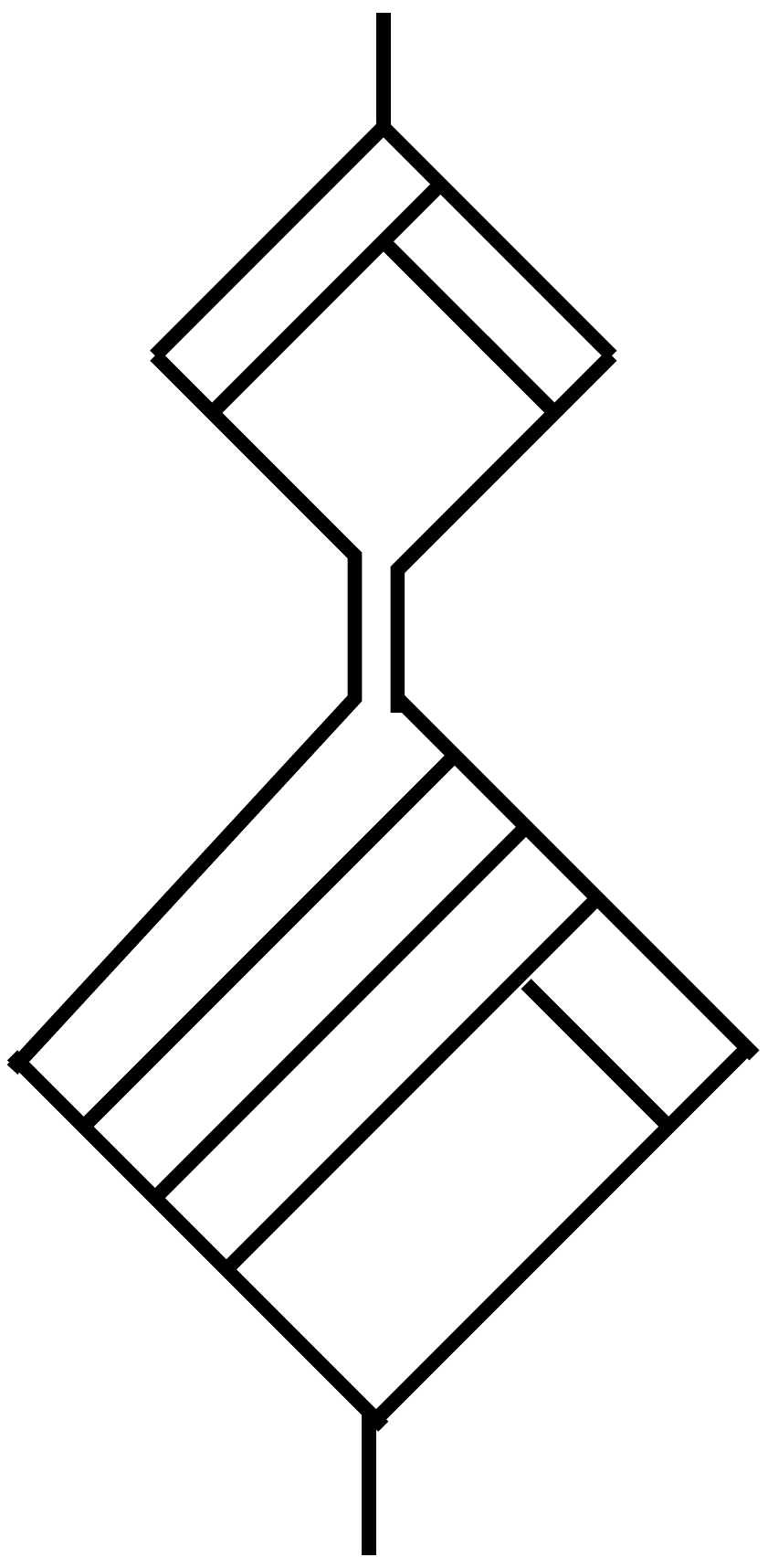} {0.8 in} $\rightarrow$ \vpic {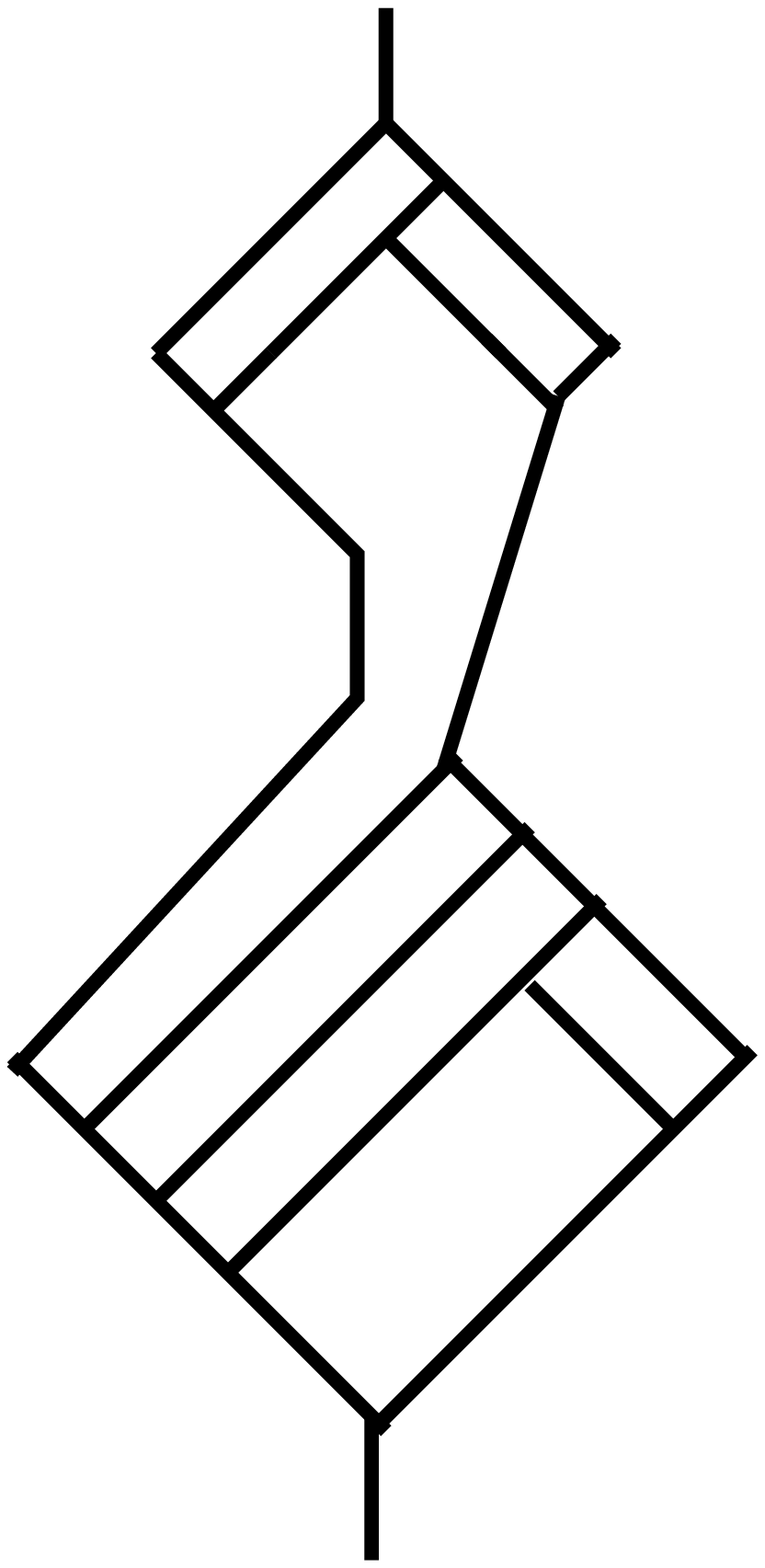} {0.8 in} $\rightarrow$ \vpic {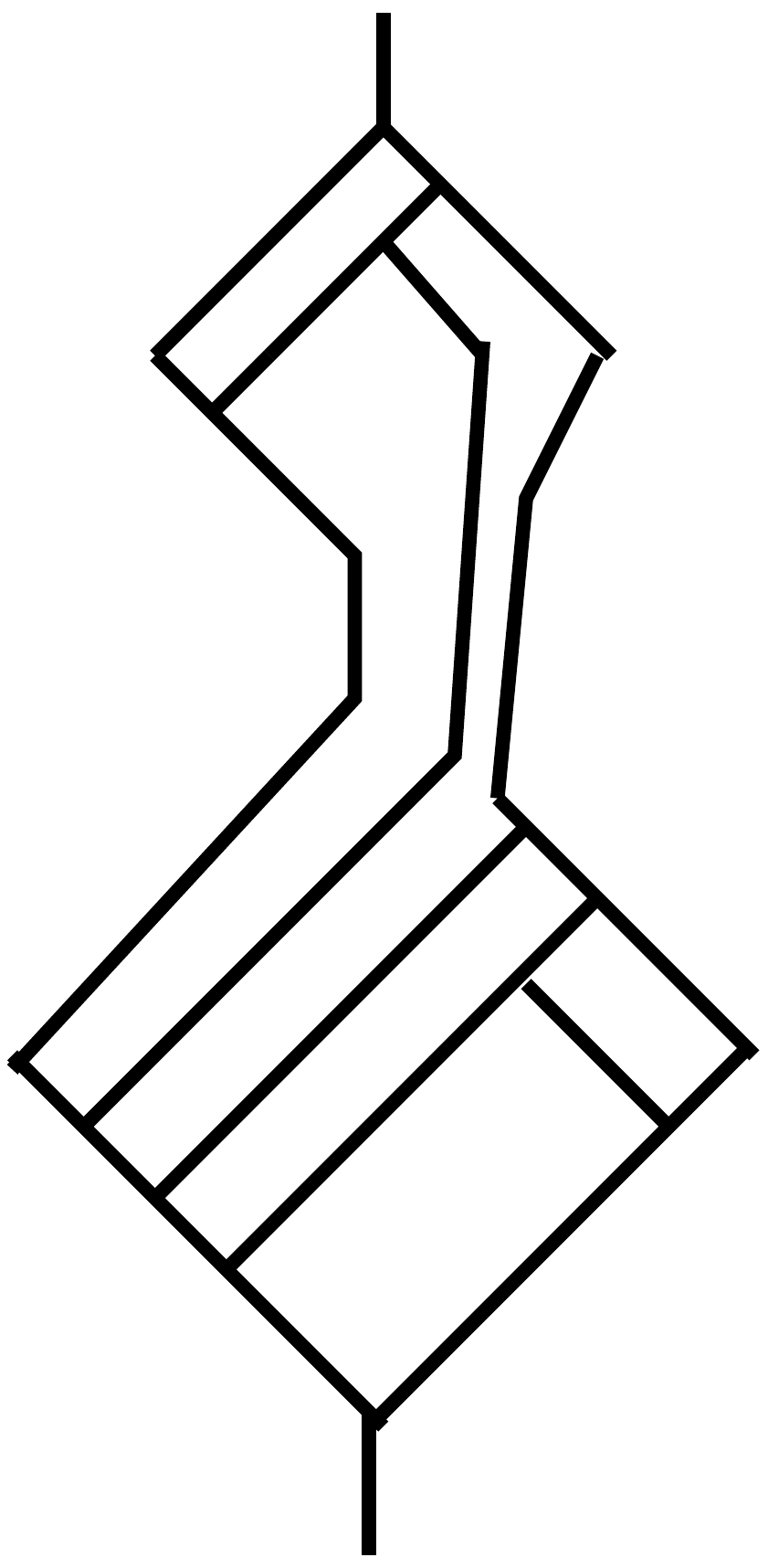} {0.8 in} 
  
  At this point one may draw a curve through the diagram so that all the vertices below it are Y's and all above it are upside down Y's.
  Straightening that curve to a horizontal line thus gives a pair of trees that give the product of the two original $F$ elements $P$ and $Q$:
  
  \vpic {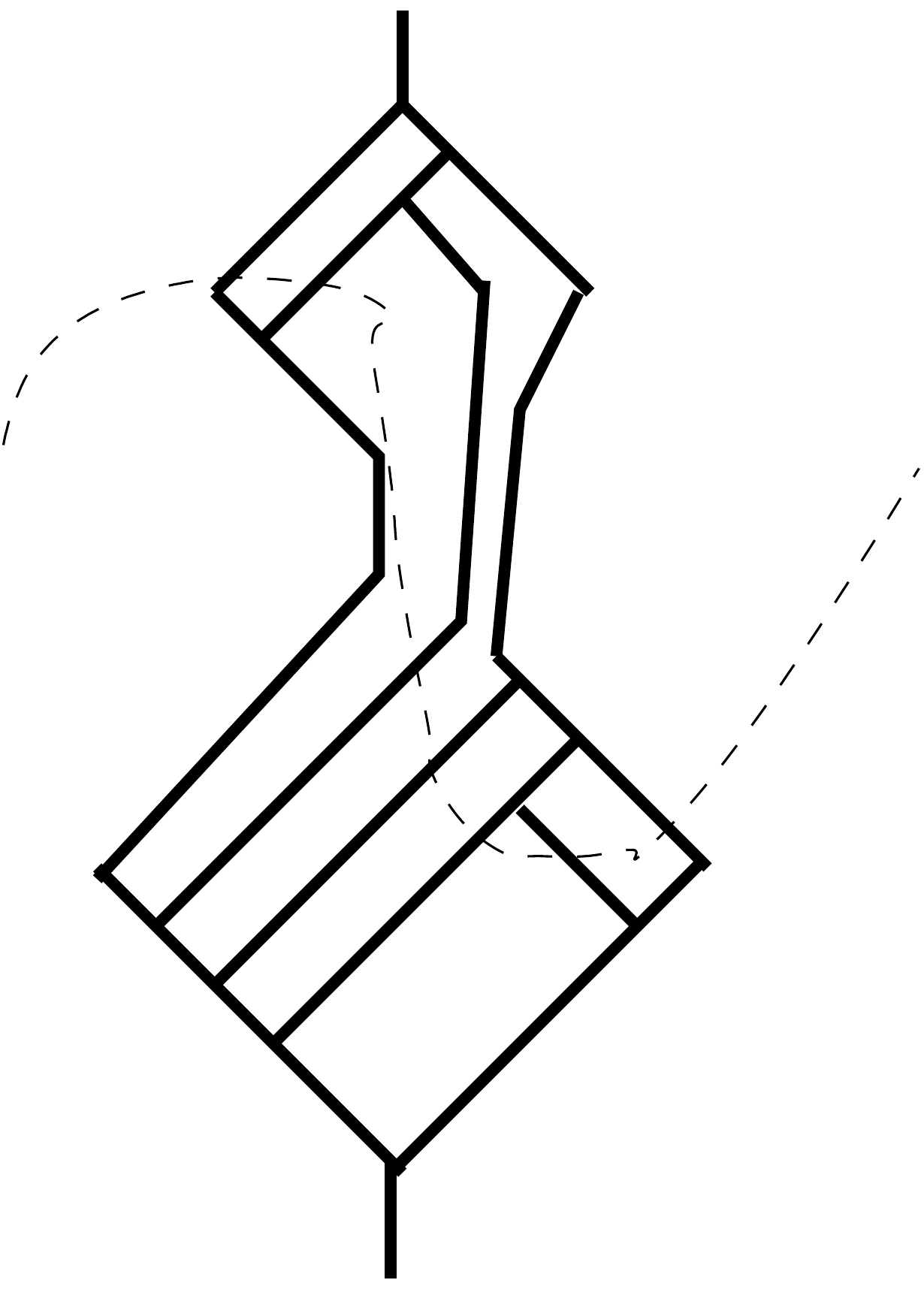} {1.5in} $\rightarrow$ \vpic {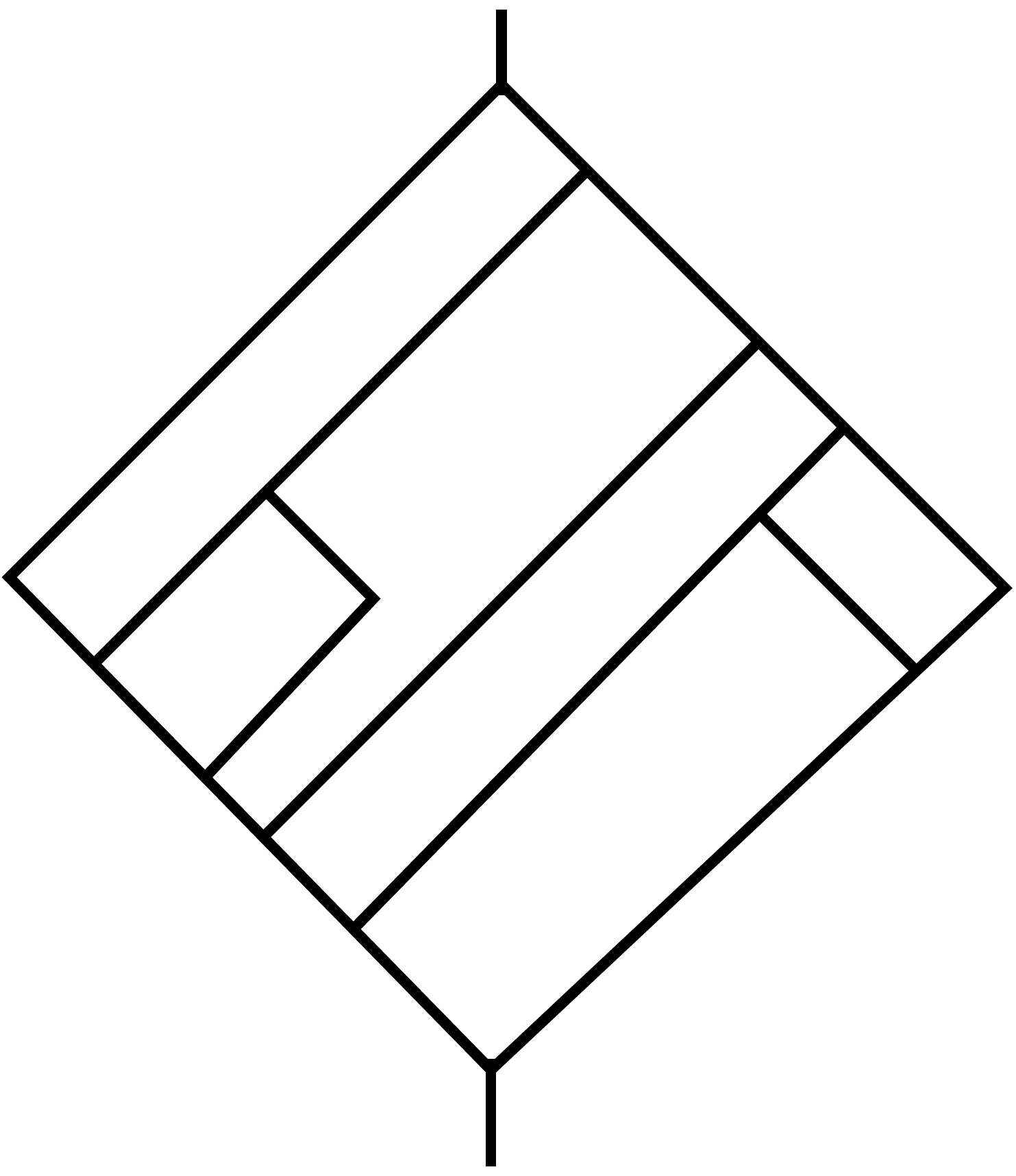} {1.5in}

\emph{NB} It is important to note that, although the diagrams involved in calculating $gh$ make
sense in the category $\mathcal C$, one may NOT use them to calculate vacuum expectation values
until all the cancellations are done. This is because the relation \vpic {rule1} {0.5in} does not hold in $\mathcal C$.

The following lemma is just a simple calculation but that calculation leads on to all the 
lucky accidents that make the proof work.Recall that $A$ is the generator of $F$ in \cite{CFP}.

\begin{lemma}
$\pi_Y(A)\Psi=\Psi$.
\end{lemma}
\begin{proof}
The two-tree picture of $A$ is  \vpic {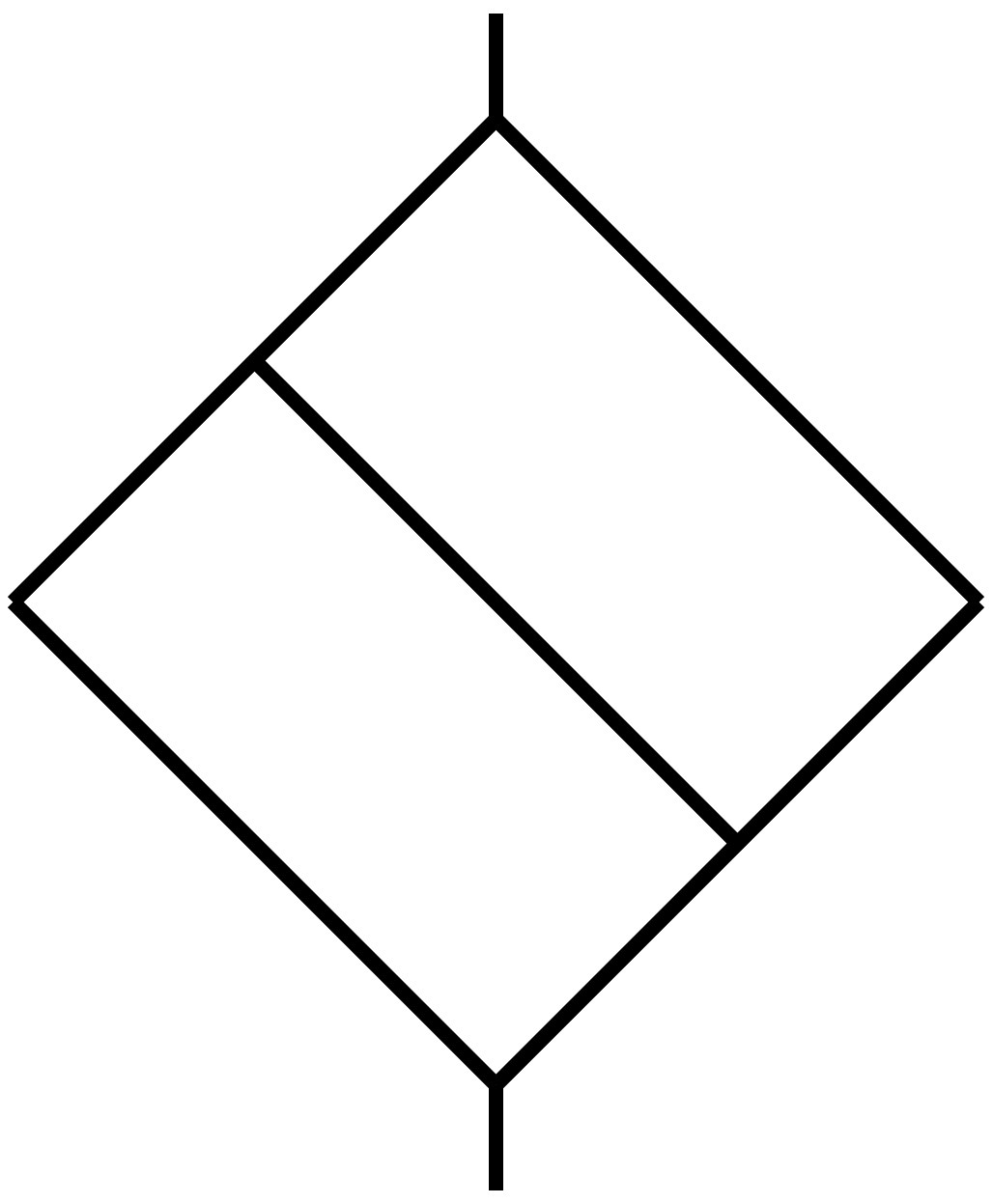} {0.5in} so the diagram in $\mathcal C$ giving 
$\langle A\Psi, \Psi\rangle$ is \vpic {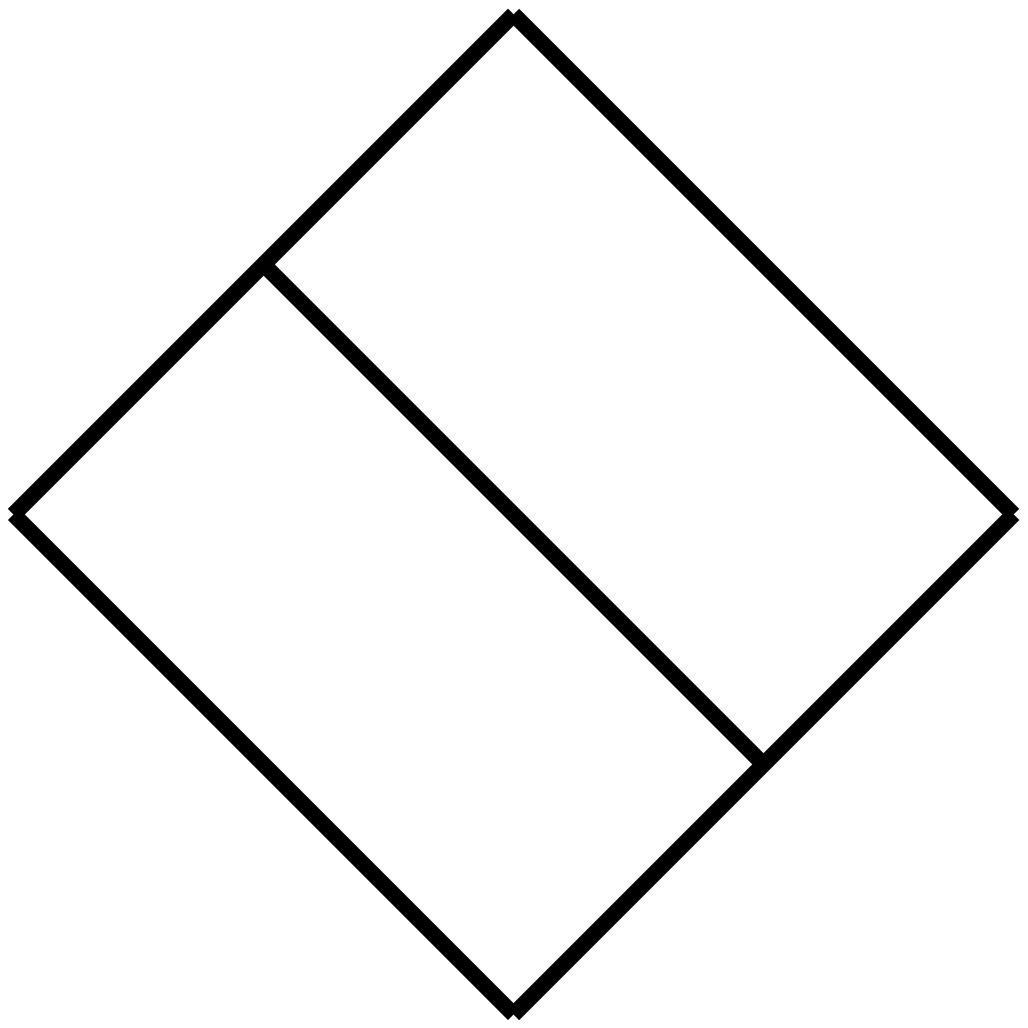} {0.5in} which by the rules in $\mathcal C$ is
the same as $d\langle \Psi, \Psi\rangle$. By the unitarity of $\pi_Y$ we are done.
\end{proof}
We leave it to the reader to deduce this result directly from the definition of the action of
group elements on vectors in the direct limit so that the result actually holds even in the 
non-unitary case.
(Note that we have suppressed the $\pi_Y$ in the inner product formula. We will continue in this way.)
\begin{lemma} Let $g$ and $h$ be in $F$. Then 
$$\langle A^ng\Psi, h\Psi\rangle=    \langle g\Psi,\Psi\rangle\langle \Psi,h\Psi\rangle$$ for sufficiently large $n$.
\end{lemma}
\begin{proof} Observe that $\langle A^ng\Psi, h\Psi\rangle=  \langle h^{-1} A^ng\Psi, \Psi\rangle$ so we must calculate the vacuum expectation
value of $h^{-1} A^ng$.

The diagram for $h^{-1}A^ng$ before cancellation is (with $n$ sloping lines in the $A^n$ part):

\vpic {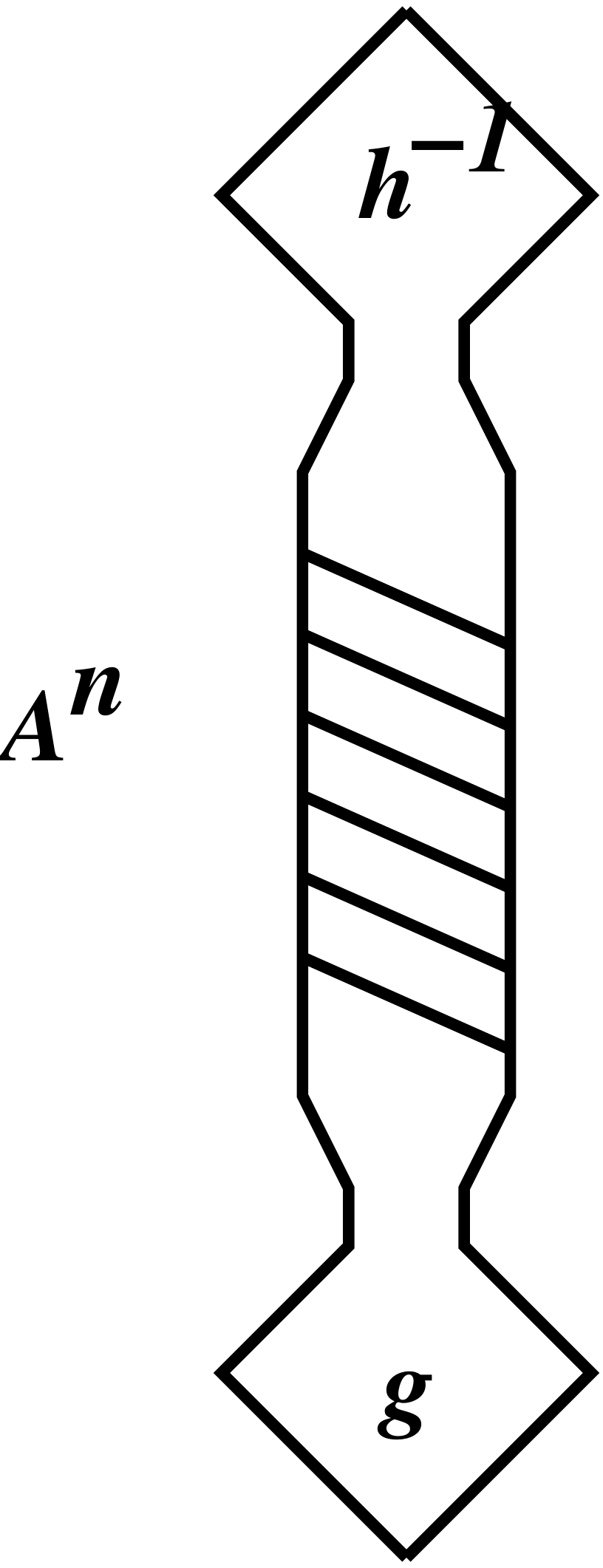} {1in}

 Now take the "pair of trees"  picture of $g$ and isolate the vertices on the  right branch of the top tree as below:

$\mbox{               } g=$ \quad \vpic {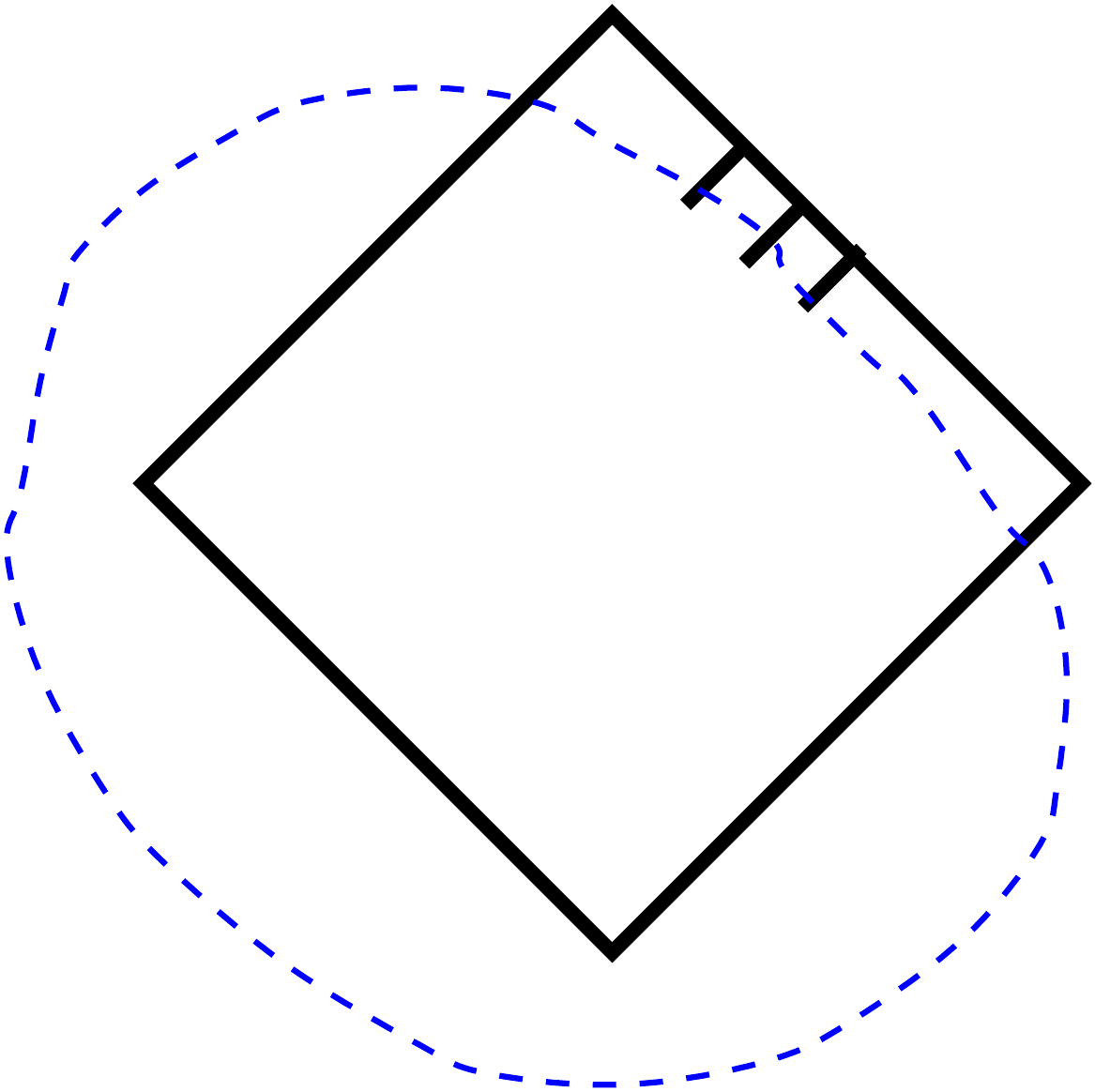} {1.5in}

Look near the top of $g$ and the bottom of $A^n$:

$\mbox{                                     } $\qquad \qquad \vpic {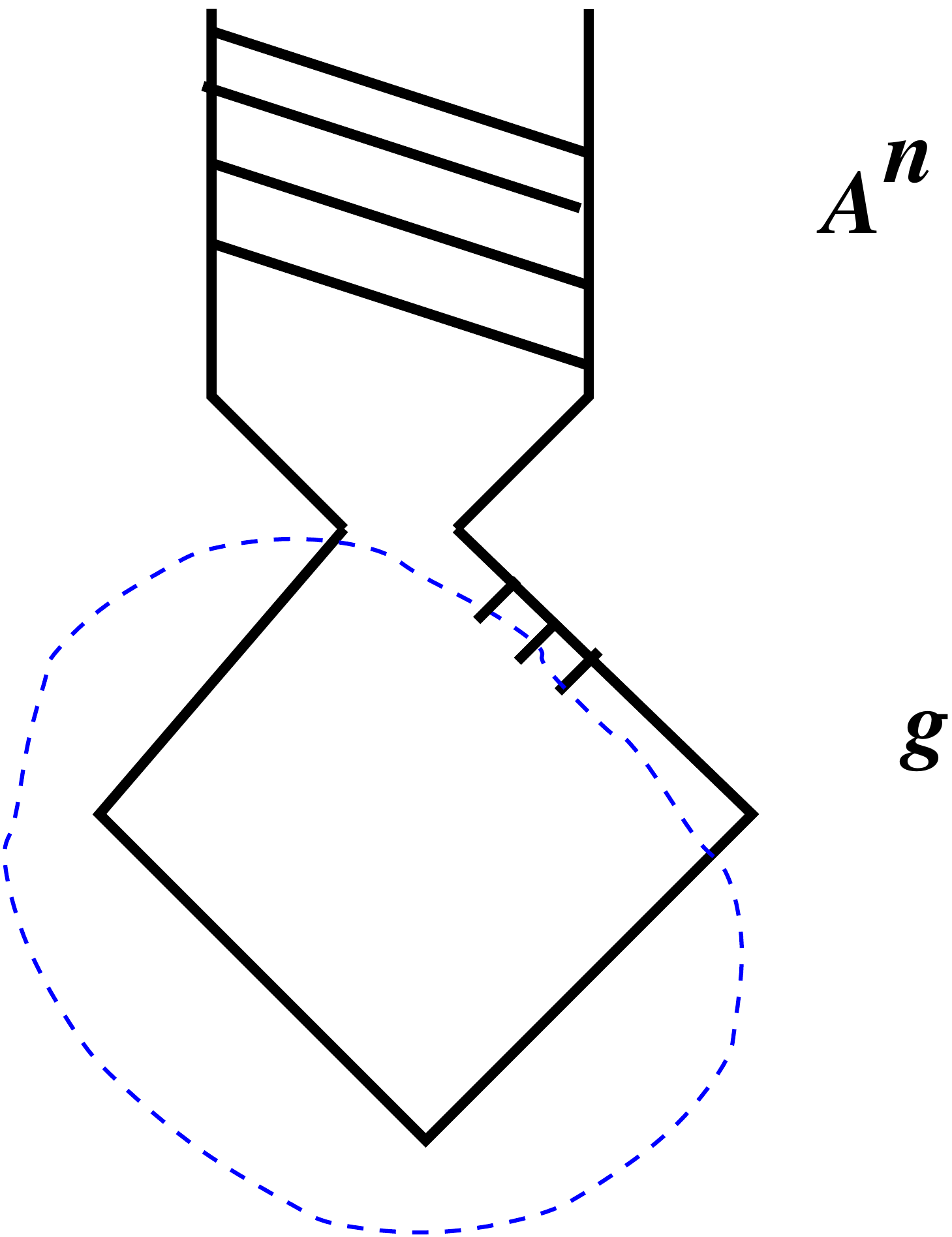} {1.5in}

 For sufficiently large $n$ one may apply \vpic {rule1} {0.5in} 
enough times to cancel all the right branch vertices in the top half of $g$ to obtain:

$\mbox{                                     } $\qquad \qquad \vpic {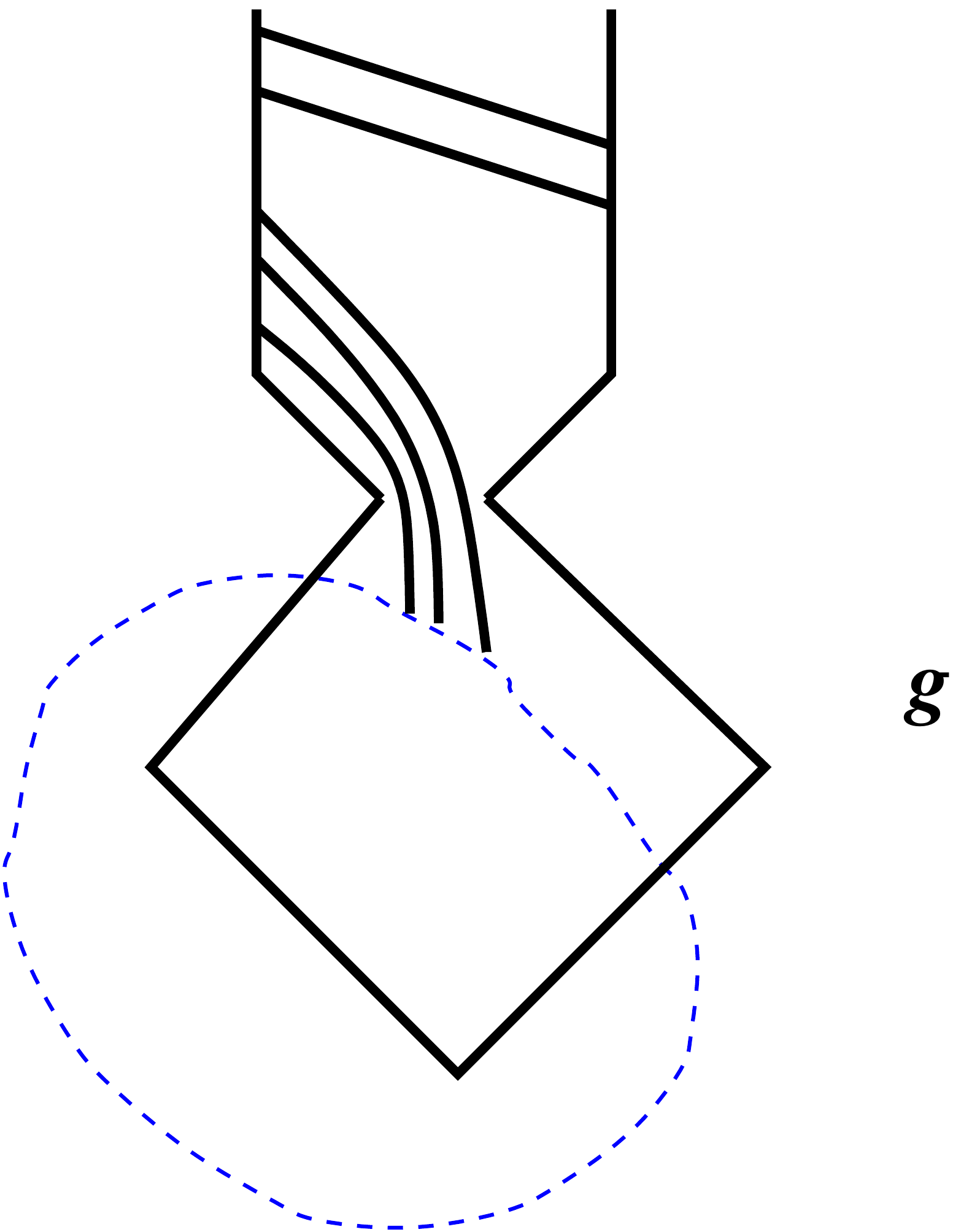} {1.5in}

At this stage all the strings connecting what is left of $A^n$ to $g$ are connected at the bottom as they were in $g$, which
is reduced. So provided $n$ is large enough so there is no interference between
$g$ and $h^{-1}$, the only cancellation that can occur in the picture is the leftmost of these strings, marked "a" below, which can only
cancel with a vertex of $A^n$ if it is connected directly to the leftmost branch of the top tree of $g$.

$\mbox{                                     } $\qquad \qquad \vpic {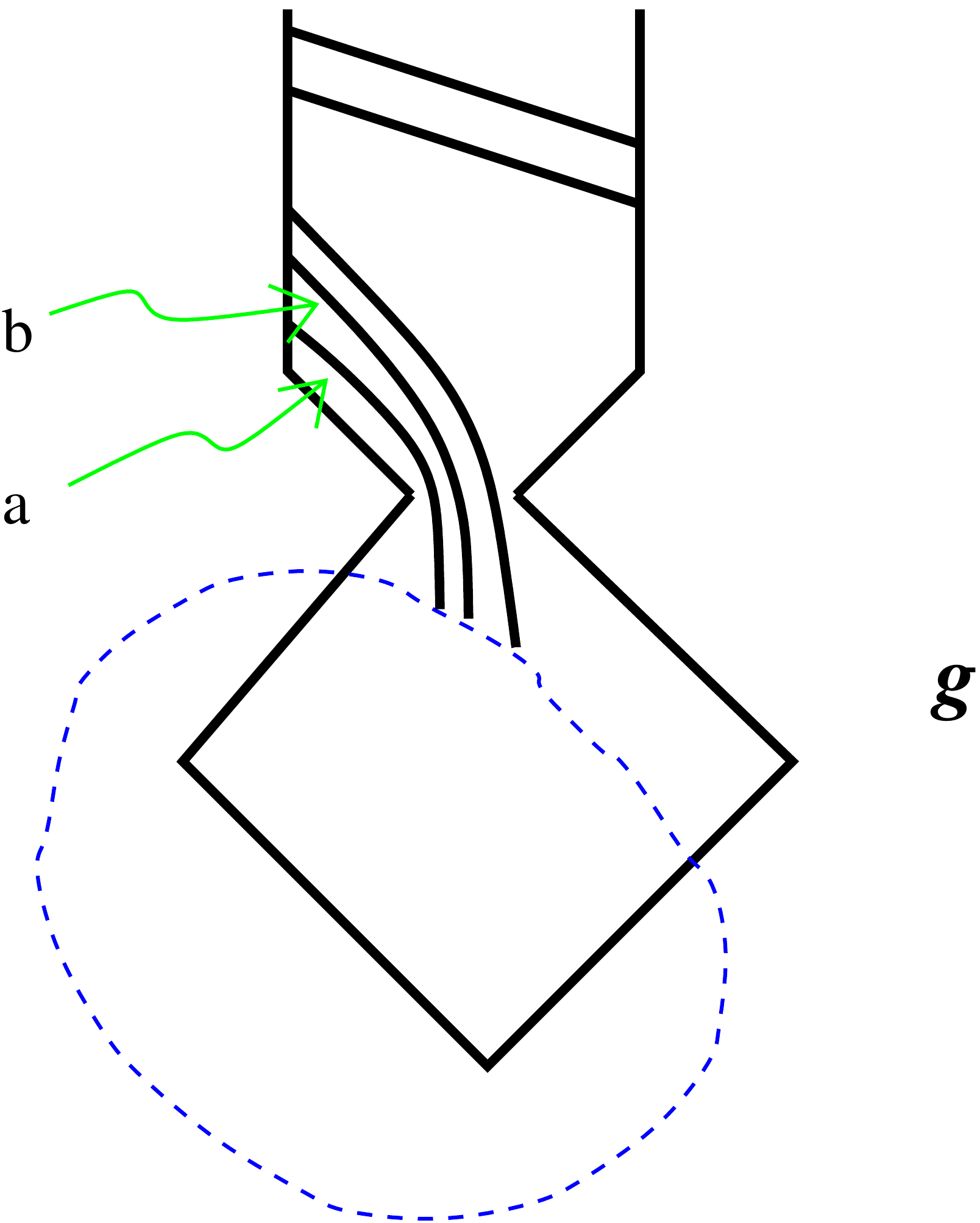} {2in}

After the string $a$ has been cancelled the only other cancellation is $b$ and so on. Once all these 
cancellations are done, no more can happen. The same considerations apply to $h^{-1}$ so we may
calculate $ \langle h^{-1} A^ng\Psi, \Psi\rangle$ by looking at the resulting figure in the trivalent 
category. But before doing that observe that the  last round of  cancellations are all of the form 
\vpic {rule2} {0.5in} which is also true in $\mathcal C$ ! So we may undo them and obtain
the picture below:

$\mbox{                                     } $\qquad \qquad \vpic {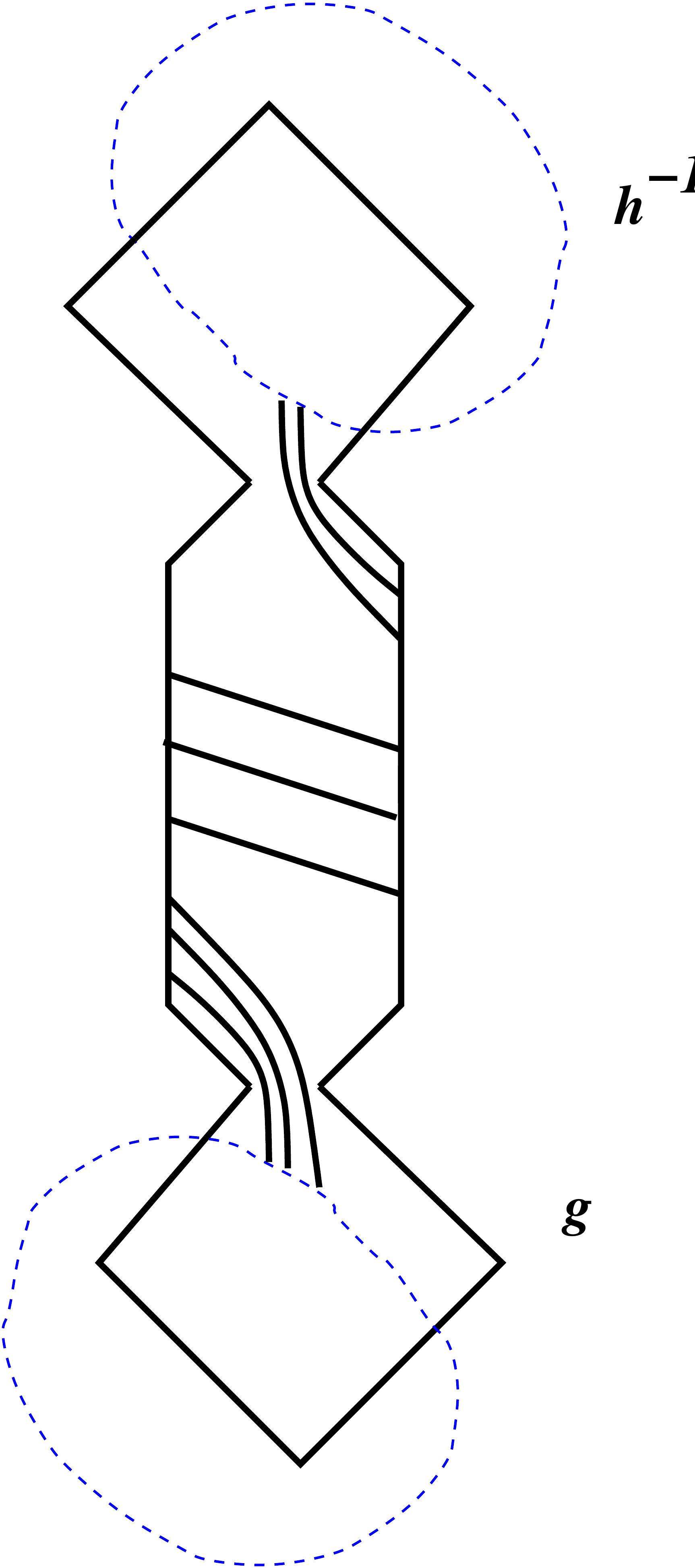} {1.5in}   what's left of $A^n$

So that  $\langle h^{-1}A^n g\Psi,\Psi\rangle$ is ${1\over d}$ times the evaluation of
the above picture in $\mathcal C$. But now isolate the parts of the picture involving $g$ and $h^{-1}$ as
below:

$\mbox{                                     } $\qquad \qquad \vpic {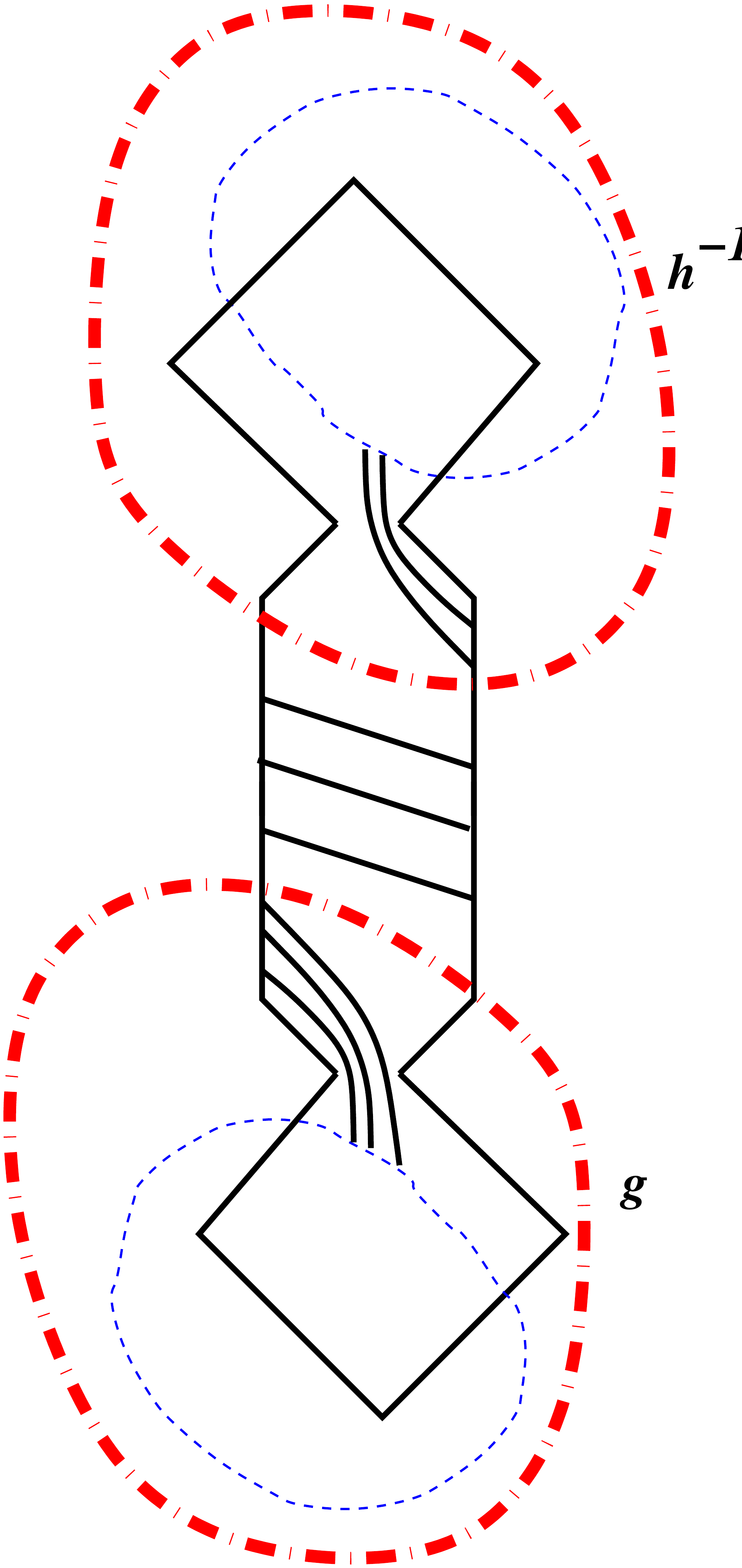} {1.5in}

The parts enclosed in the red double dashed regions are actually elements of $\mathcal C_2$ which is one dimensional.
Hence they are scalar multiples of a single curve joining the boundary points.  
To find the scalar multiple cut out the parts inside the red dotted lines and join the boundary points to obtain (for $g$, $h$ is similar) :

G=$\mbox{                                     } $\qquad \qquad \vpic {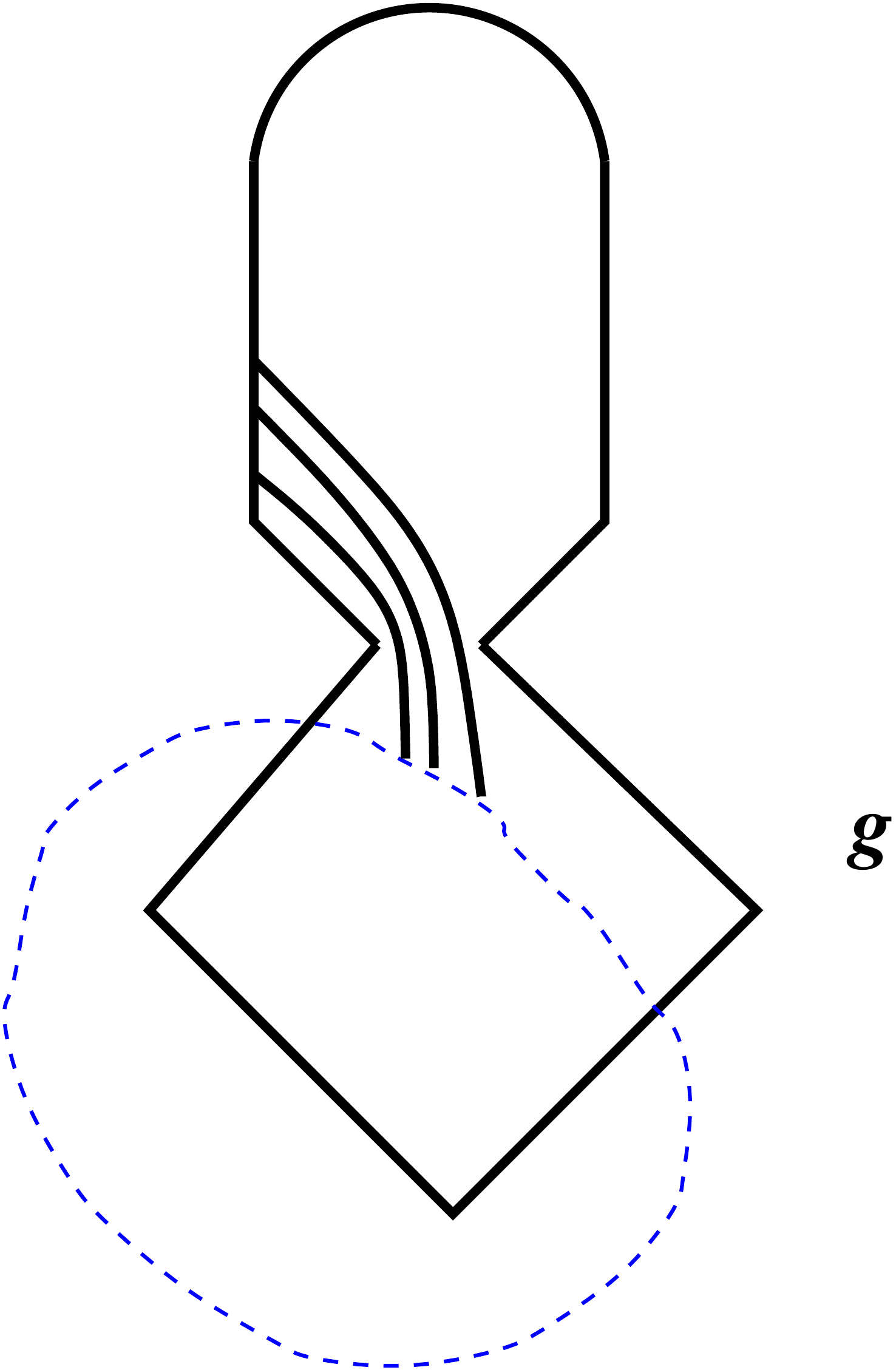} {1.5in}

The curves which started out connected to the top right branch of $g$, can be moved back there, the top of the picture
straightened, and we recognise the picture of $g$ that we started with! And similarly a factor $H$ at the top.
 By wysiwyg 
 we have $$G=d \langle g\Psi,\Psi\rangle \mbox{  and  } H=d\langle h^{-1}\Psi,\Psi\rangle.$$
 By easy planar algebra arguments,
splitting a "connected sum" of two closed tangles as the product of individual closed tangles incurs a multiplicative factor of
${1\over d}$.

  The three pictures resulting from the red double dash decomposition are $G$, $H$ and  $A^k$ for some $k$ with $0\leq k\leq n$, to which we may
apply \vpic {rule2} {0.5in} $k$ times  to obtain just $d$.

Altogether $d\langle A^ng\Psi, h\Psi\rangle= {d\over d^2}GH$
so $\langle A^ng\Psi, h\Psi\rangle=\displaystyle{{G\over d}{H\over d}}= \langle g\Psi,\Psi\rangle\langle \Psi,h\Psi\rangle$.

It is interesting that the same calculation works for $\Phi^X$ right up to this point. But because of the
extra string joining the top to the bottom, a multiplicative factor of $\displaystyle (\frac{d-2}{d-1})^k$ appears which means
$A^n$ actually tends weakly to zero. Indeed the projection onto $\mathbb C \Omega$ does not even commute
with $A$ in this case since if it did, $|\langle A\Omega,\Omega\rangle|$ would be $1$.
\end{proof}
\begin{corollary}
The weak limit of $\pi_Y(A^n)$ as $n\rightarrow \infty$ on the closure of the $F$-linear span of $\Psi$ is equal to $P_\Psi$, the orthogonal 
projection onto the linear span of $\Psi$.
\end{corollary}
\begin{proof} The formula for $P_\Psi$ is $P_\Psi(\eta)=\langle \Psi,\eta\rangle \Psi$. Linear extension of the result of the
previous lemma gives  $\langle A^n\xi, \eta\rangle=  \langle P_\Psi \xi,\eta\rangle$ for $\xi$ and $\eta$ in a dense subspace 
and the corollary follows from the unitarity of $\pi_Y$.
\end{proof}

\begin{proof} The proof of the theorem is now easy. For let $\mathfrak H $ be the closed $F$-linear span of $\Psi$ on which $\pi_Y$ acts.
Suppose $t$ is a bounded operator on $\mathfrak H$ commuting with $\pi_Y(F)$. Then by the previous corollary $t\Psi=\lambda\Psi$
for some scalar $\lambda$. Thus $tg\Psi=\lambda g\Psi$ for all $g\in F$ so taking linear combinations we see that $t=\lambda id$
on all of $\mathfrak H$.
\end{proof}
\begin{corollary}The Wysiwyg representations from $\Phi^1$ are mutually inequivalent for different values of $d$.
\end{corollary} 
\begin{proof} The vector $\Psi$ when normalised is unique up to a scalar of modulus one, as a vector spanning the
1-eigenspace of A. But we can calculate, using the relations of \cite{MPS}, $\langle X\Psi,\Psi\rangle=d\frac{d-2}{d-1}$. But $\langle \Psi,\Psi\rangle=d$
so two Wysiwyg representations can only be equivalent if the values of $d$ are equal.

\end{proof}

\begin{remark} We know that the inductive limit Hilbert space also carries a unitary representation of Thompson's group $T$, extending
that of $F$ and also called $\pi_Y$ Exactly the same argument as above shows that $T$ acts irreducibly on the closed $T$-linear 
span of $\Psi$ but we do not know if that is the same Hilbert space as $\mathfrak H$.
\end{remark}

\section{Improvements}

We present two improvements on the previous results by extending from the $F$-linear span of the vacuum vector to 
the whole direct limit Hilbert space.

a) Telling the representations apart. 

It is curious that all the representations of $F$ the kind we are considering contain the coefficients of the vacuum 
by weak limits of certain elements. Which means that the representations are inequivalent as soon as the corresponding
vacuum ones are.

To be more precise let $\sigma:F\rightarrow F$ be the endomorphism which sends an element of $F$ (viewed as a 
PL homeomorphism of $[0,1]$) to itself rescaled and acting in $[\frac{1}{2},1]$ and by the identity on $[0,\frac{1}{2}]$.

Diagrammatically we have, for $g=\frac{a}{b}$ :

$$\sigma(g)=  \vpic {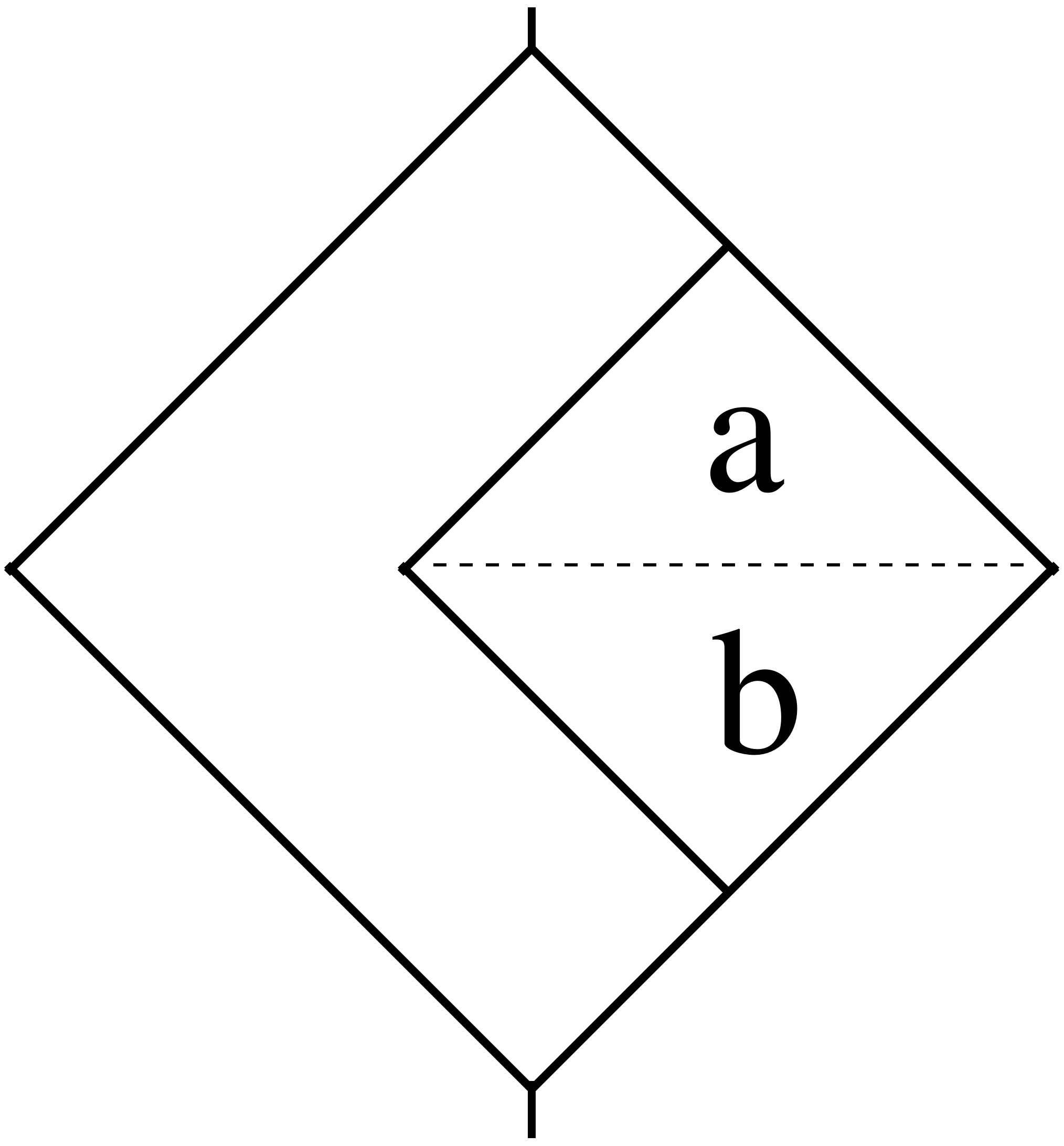} {1in}  $$

\begin{theorem} Let $g\in G$ and 
Then $$\pi_{Y,p}(\sigma^n(g)) \mbox{tends weakly to } \langle \pi (g) \Omega,\Omega\rangle id \mbox{ for any }  p \mbox{ as in } \ref{proj}$$
\end{theorem}
\begin{proof} 

By unitarity as usual it suffices to show that
 $$\langle \pi_{Y,p}(\sigma^n(g))(\xi),\eta\rangle \rightarrow  \langle \pi (g) \Omega,\Omega\rangle\langle \xi,  \eta \rangle$$
for $\xi$ and $\eta$ in any finite dimensional approximant of the direct limit.  Let $t_m$ be the full bifurcating planar tree with $2^m$
leaves (such trees form a cofinite sequence) and let $x$ and $y$ be elements of $\mathcal C_{2^m}$. Then let 
$\displaystyle \xi=\frac{t_m}{x}$ and $\displaystyle \eta=\frac{t_m}{y}$ be the corresponding elements of the direct limit Hilbert space. We need to calculate
$\langle \pi_{Y,p}(\sigma^n(g))(\xi),\eta\rangle$ for large $n$. For any $k>0$ and any tree $s$ let $r^k_s$ be the tree with $n$ branches 
to the left attached to a single branch to the right, and $s$ attached to the end of the right branch, thus:

$$r^3_s = \vpic {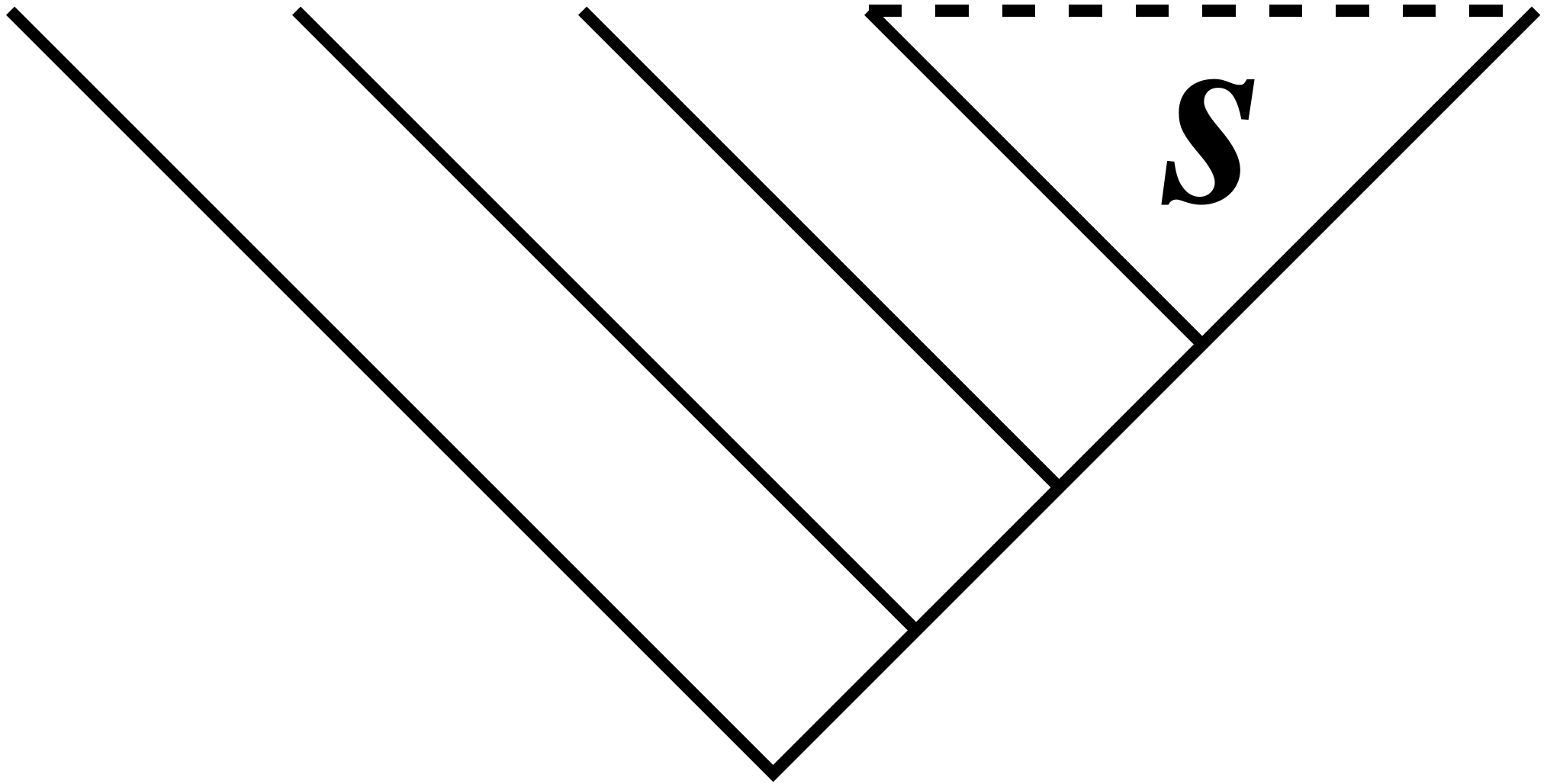} {1.5in}  $$

Suppose $n>m$ and that $a$ is a tree. Let $t_a$ be the tree obtained by attaching $r^{n-m}_a$ to the rightmost leaf of $t_m$. 
  Let $g$ be given by the pair of trees $\displaystyle \frac{a}{b}$. Then by drawing a diagram and cancelling lots of carets we see that $\displaystyle \frac{t_a}{t_b}=\sigma^n(\frac{a}{b})=\sigma^n(g)$. Moreover $$\frac{t_m}{x}=\frac{t_b}{ x_b}$$ where $ x_b$ is the element of $\mathcal C$ defined
  by attaching the diagram $r^{n-m}_b$ to the rightmost string emanating from a disc containing $x$ thus (illustrated for $m=3$ so there are 8
  vertical strings emanating from the disc containing $x$):
  
  $$x_b = \vpic {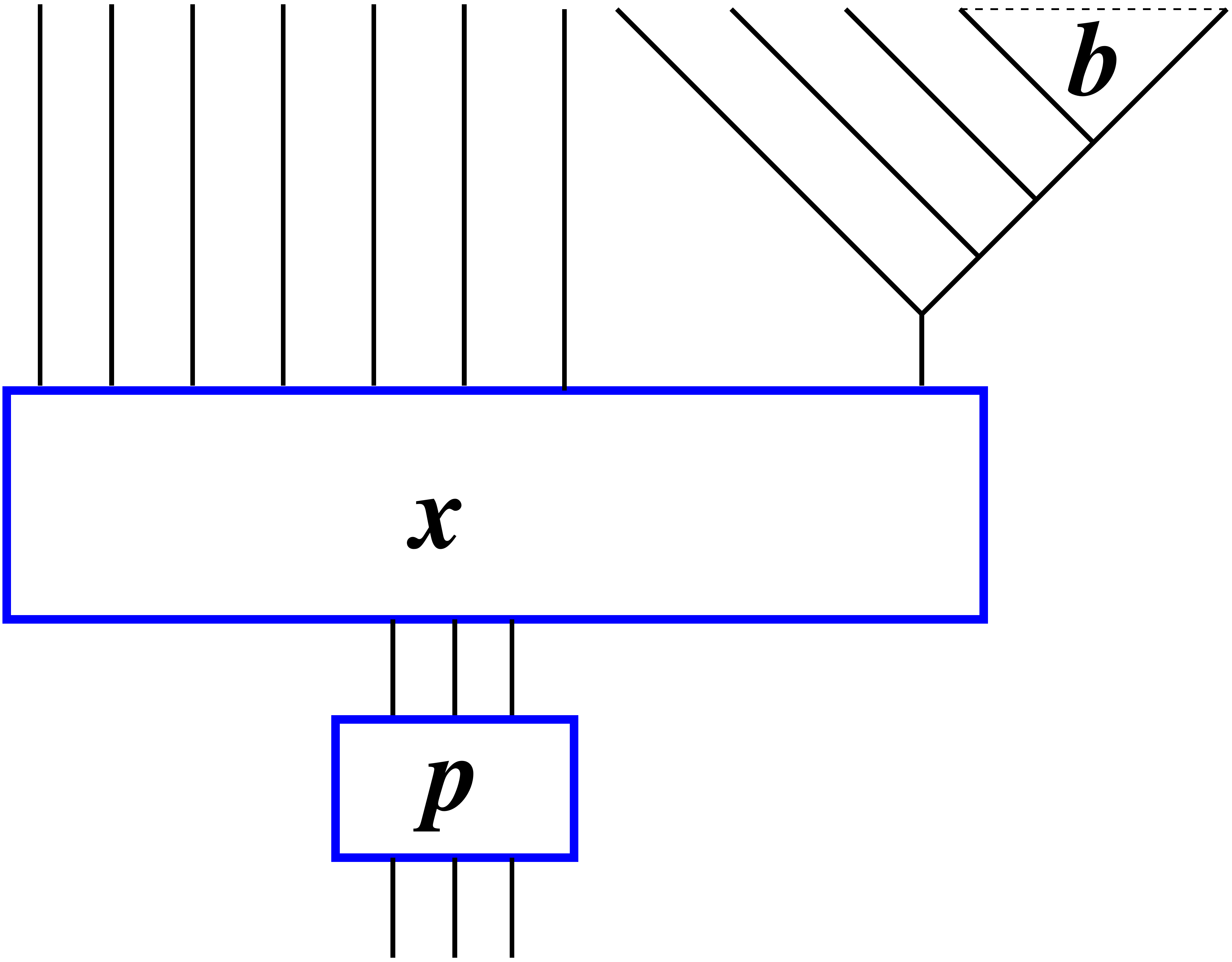} {2in}  $$
  
  This picture is interpreted as an element of the planar algebra with the trivalent vertices being discs containing $Y$, as usual.

We have $$\sigma^n(g)(\xi)=\frac{t_a}{t_b}(\frac{t_b}{x_b})=\frac{t_a}{x_b}$$

and of course $$\eta=\frac{t_m}{y}=\frac{t_a}{y_a}$$ so working the $\mathcal C$ we see that 
$\displaystyle \langle \pi_{Y,p}(\sigma^n(g))(\xi),\eta\rangle$ is the value in $\mathcal C$ of the following diagram (illustrated for $n=6$):

 $$ \vpic {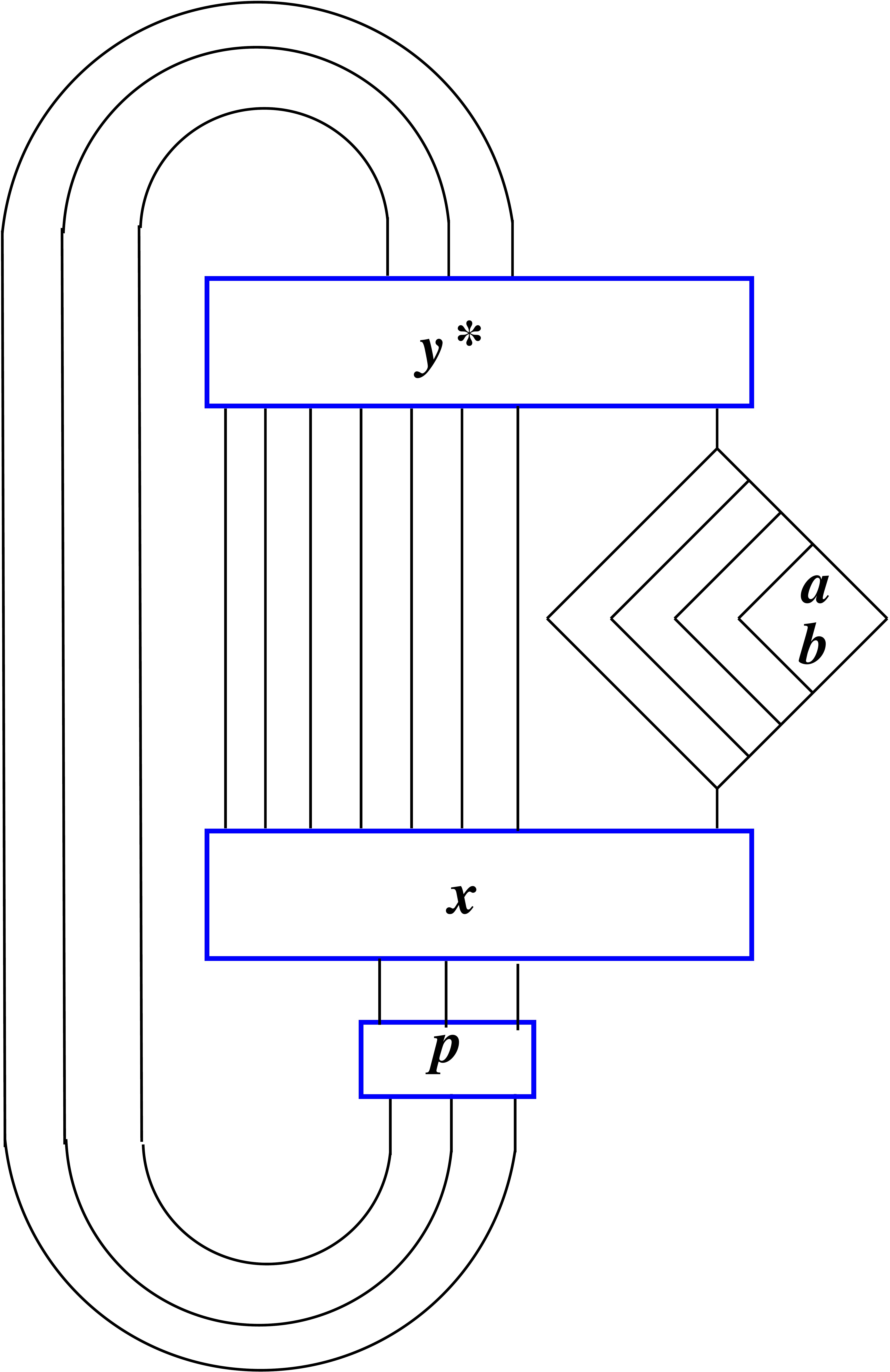} {2in}  $$
 
 Since the dimension of $\mathcal C_1$ is one, we may surround the diamond shape part of the picture by a disc and use the unitarity
 of $Y$ to
 conclude that this picture is a multiple of $\langle \pi (g) \Omega,\Omega\rangle\langle \xi,  \eta \rangle$.
\end{proof} 
\begin{corollary} 
$\pi_{Y,p}$ is inequivalent to $\pi_{Y',q}$ whenever $\pi_{Y,1}$ is inequivalent to $\pi_{Y',1}$. 
\end{corollary}
b)In the last section we showed that $A^n$ tends weakly to the projection onto $\mathbb C \Psi$ on the $F$-linear span
of $\Psi$. Here we extend that to the whole space $\mathfrak H^1_Y$. We record also the corresponding result for
$\pi_{Y,1}$.
\begin{theorem}
(i) $\pi_{Y,\phi}(A^n)$ tends weakly to the projection onto $\mathbb C \Psi$ as $n\rightarrow \infty$.

\qquad\qquad\quad\quad\hspace{3pt} (ii) $\pi_{Y,1}(A^n)$ tends weakly to zero as $n\rightarrow \infty$.
\end{theorem} 

\begin{proof}
(i) The proof will involve a calculation similar to that of the previous theorem but fortunately the detailed structure of the
stabilising forests need not concern us. This is because we know already that there is a fixed vector for $\pi_{Y,\phi}(A)$ so 
if we can show the weak limit of $\pi_{Y,\phi}(A^n)$ exists and has rank 1, it can only be orthogonal projection onto the
subspace spanned by the fixed point. 

So let $a_n$ and $b_n$ be the trees with $n$ leaves illustrated below for $n=7$:

$$a_7= \vpic {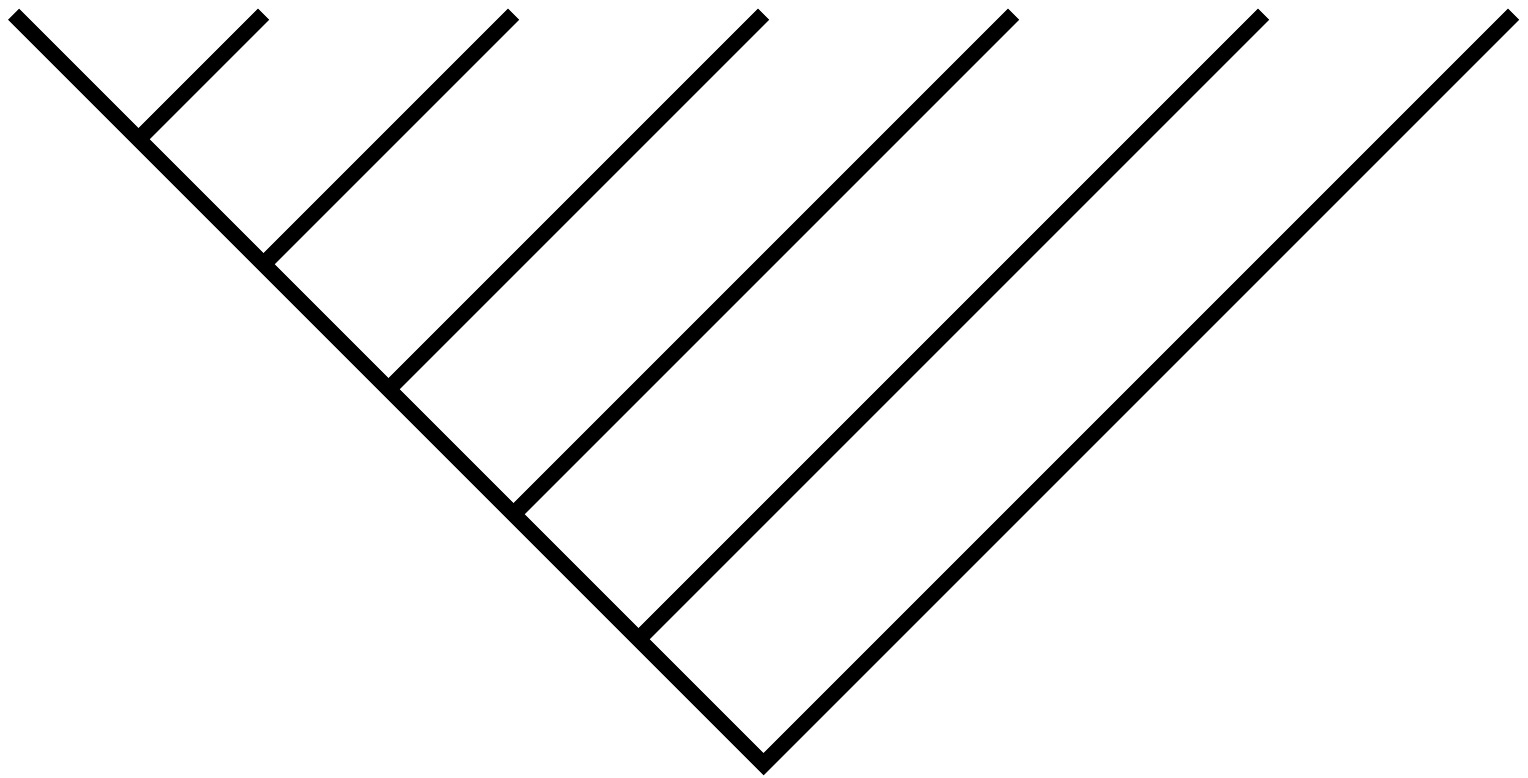} {1in}     \qquad  b_7= \vpic {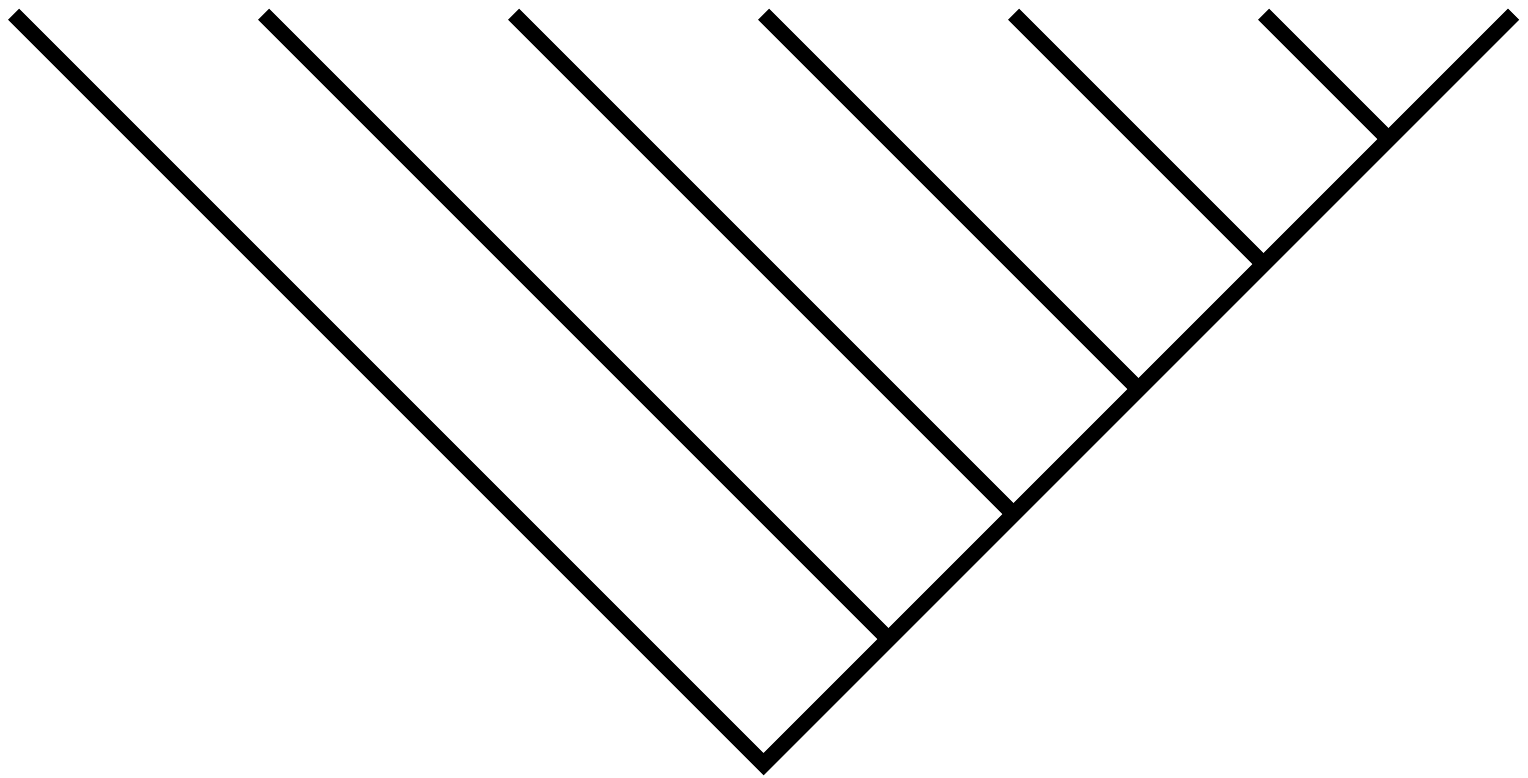} {1in}  $$

Then $\displaystyle A^n= \frac{a_{n-2}}{b_{n-2}} .$

Suppose we are given $\displaystyle \xi= \frac{t_m}{x}$ and $\displaystyle \eta= \frac{t_m}{y} $ as in the last theorem in the direct limit
Hilbert space on which $\pi_{Y,\phi}(A^n)$ acts. As before our first job is to stabilise $b_n$ and $t_m$ so they
are equal. Provided $n$ is much larger than $2^m$ this will be achieved for $b_n$ by a forest of the form:
$$ \vpic {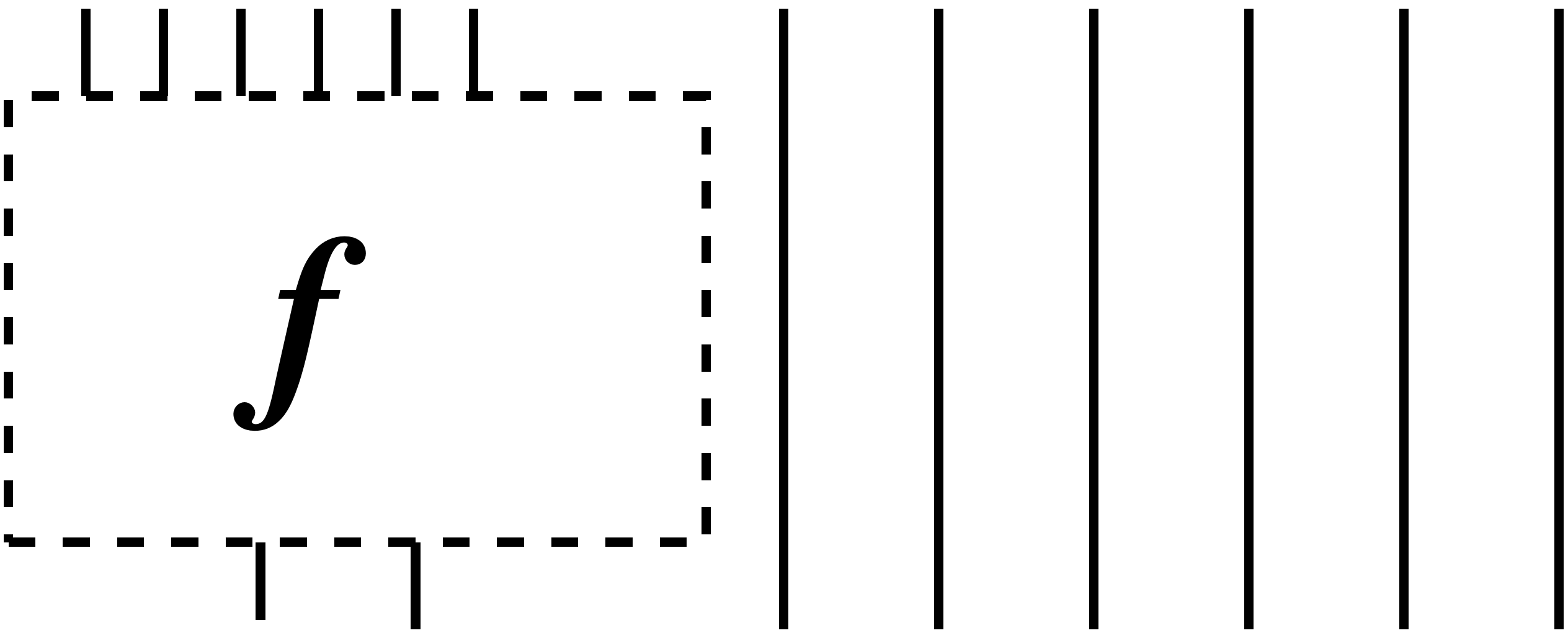} {1.5in} $$
where $f$ is the  forest  with the $m-2$ trees  $t_{m-1}, t_{m-2},\cdots$ from left to right. To stabilise the numerator   $t_m$ of $\xi$,  simply  attach a copy of $b_k$ for some large $k$ (approximately
equal to $n-2^m$)  to the rightmost leaf of $t_m$. Thus we have $$ \pi_{Y,\phi}(A^n)(\xi)=\frac{\tilde a_n} {\tilde x} $$ with
$\tilde x$ given by the following tangle (illustrated with $m=3$ and $k=7$:
$$\tilde x = \vpic {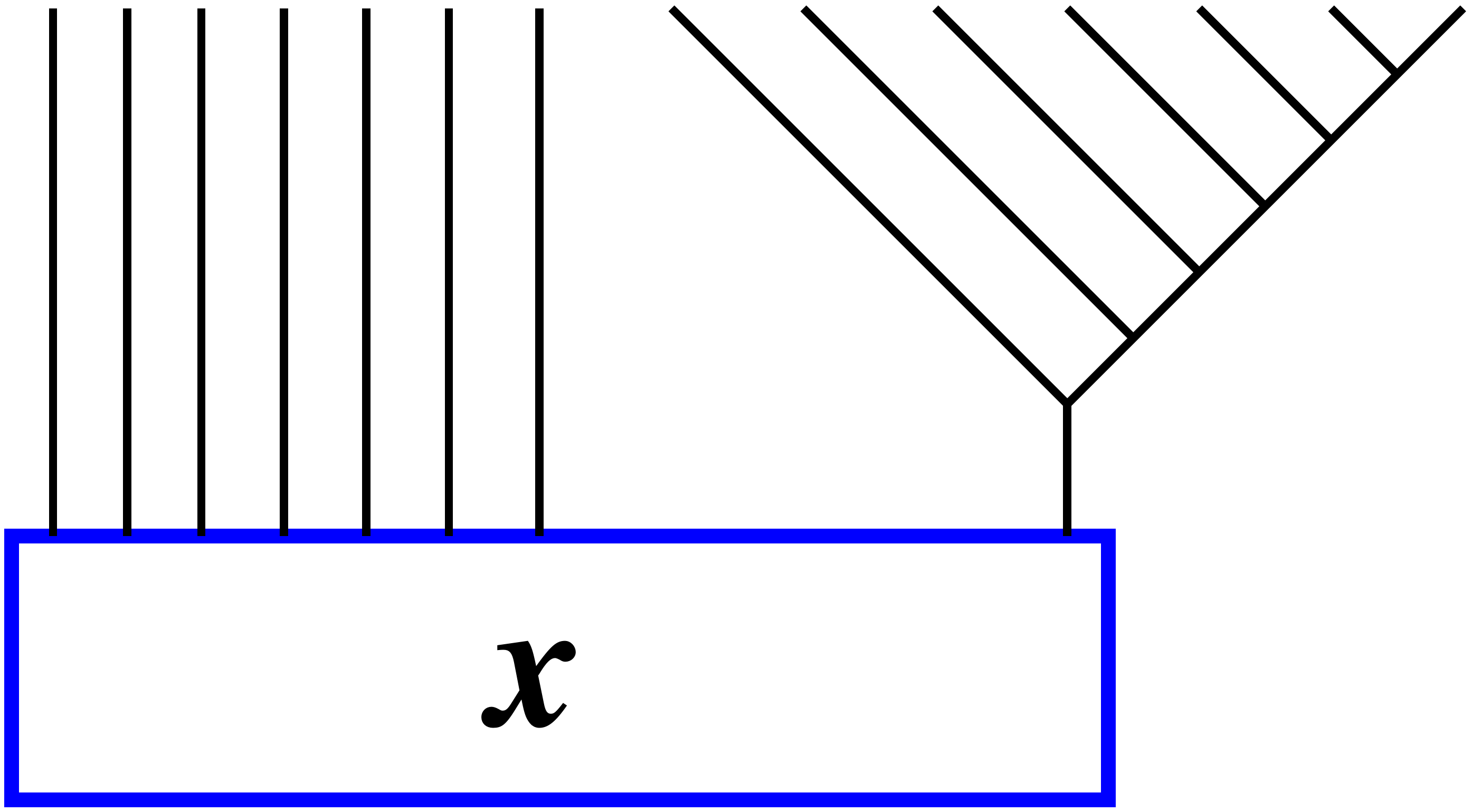} {1.5in} $$
and $\tilde a_n$ is the tree $a_n$ stabilised by the forest above thus:
$$ \vpic {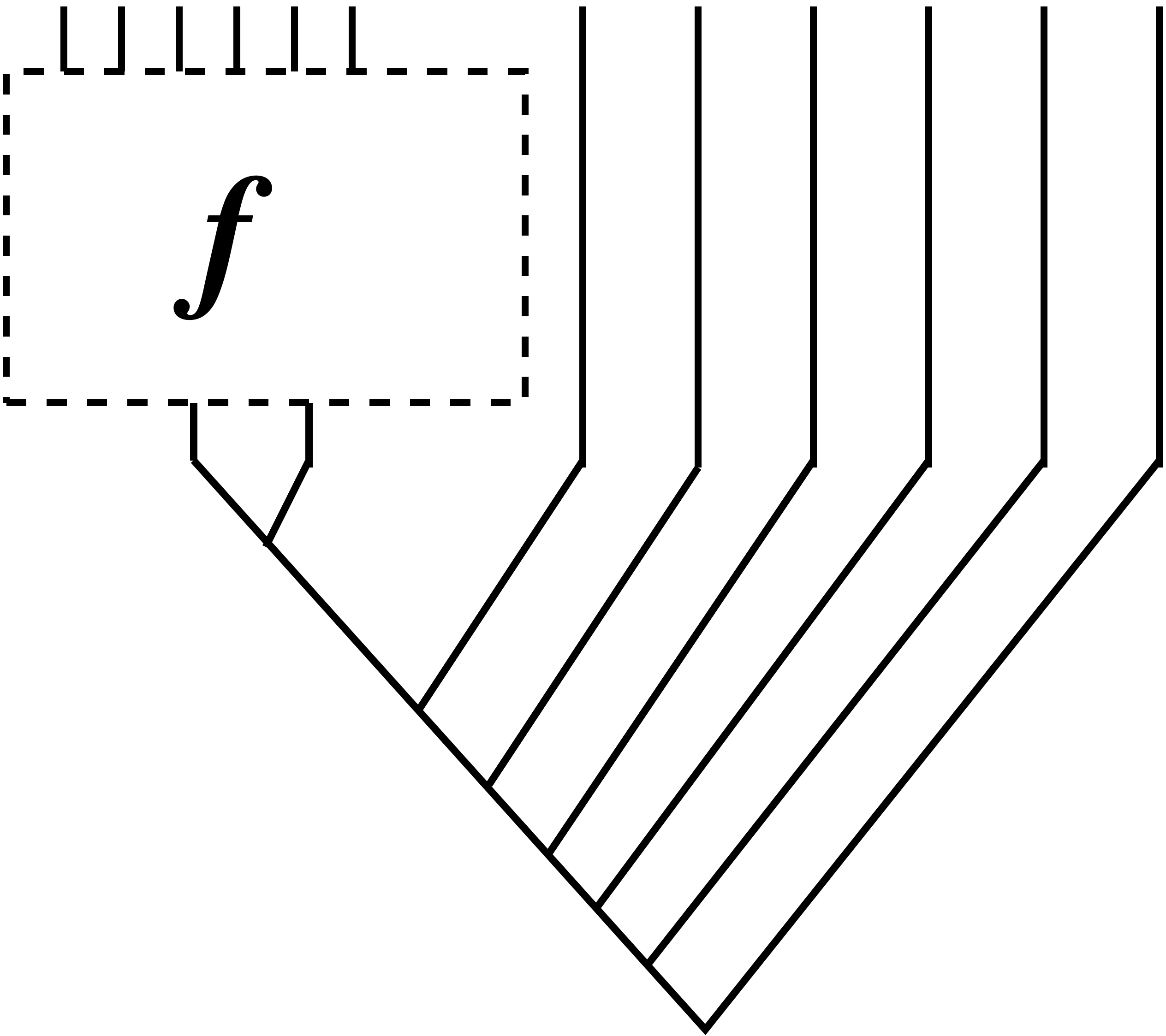} {1.5in} $$
In order to calculate the inner product $\langle \pi_{Y,\phi}(A^n)(\xi),\eta\rangle$ we need to stabilise $\tilde a_n$ and 
$t_m$ so that they are equal. This involves attaching a forest $g$ which is a reflected version of $f$ on the right hand side of $\tilde a$, 
and attaching a copy of $\tilde a_k$ for some $k$ to the left hand side of $t_m$. Applying these stabilisations to the denominators
we see that we want the inner product in $\mathcal C$ of 
$$ \vpic {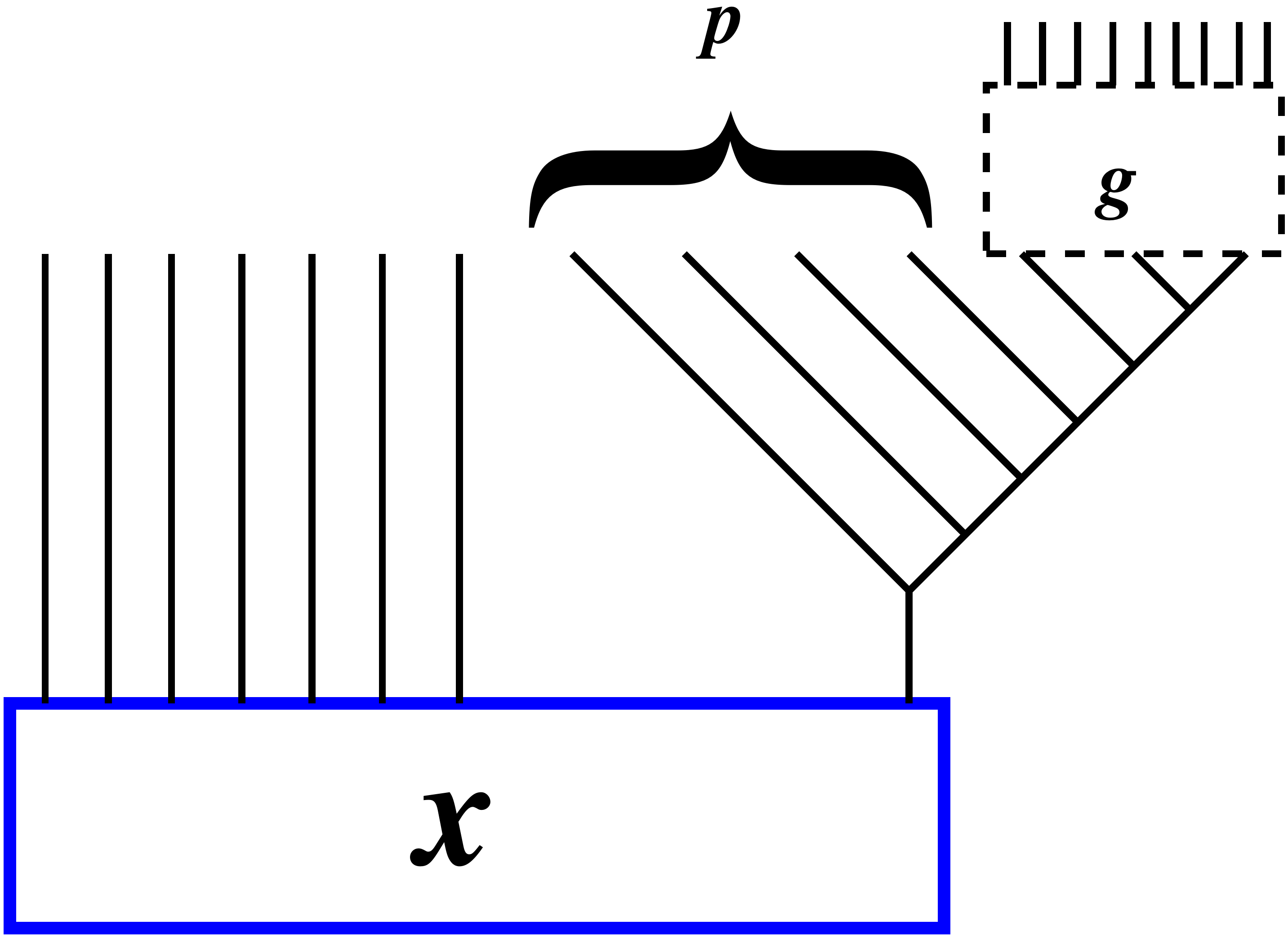} {2in} \mbox{  and  } \vpic {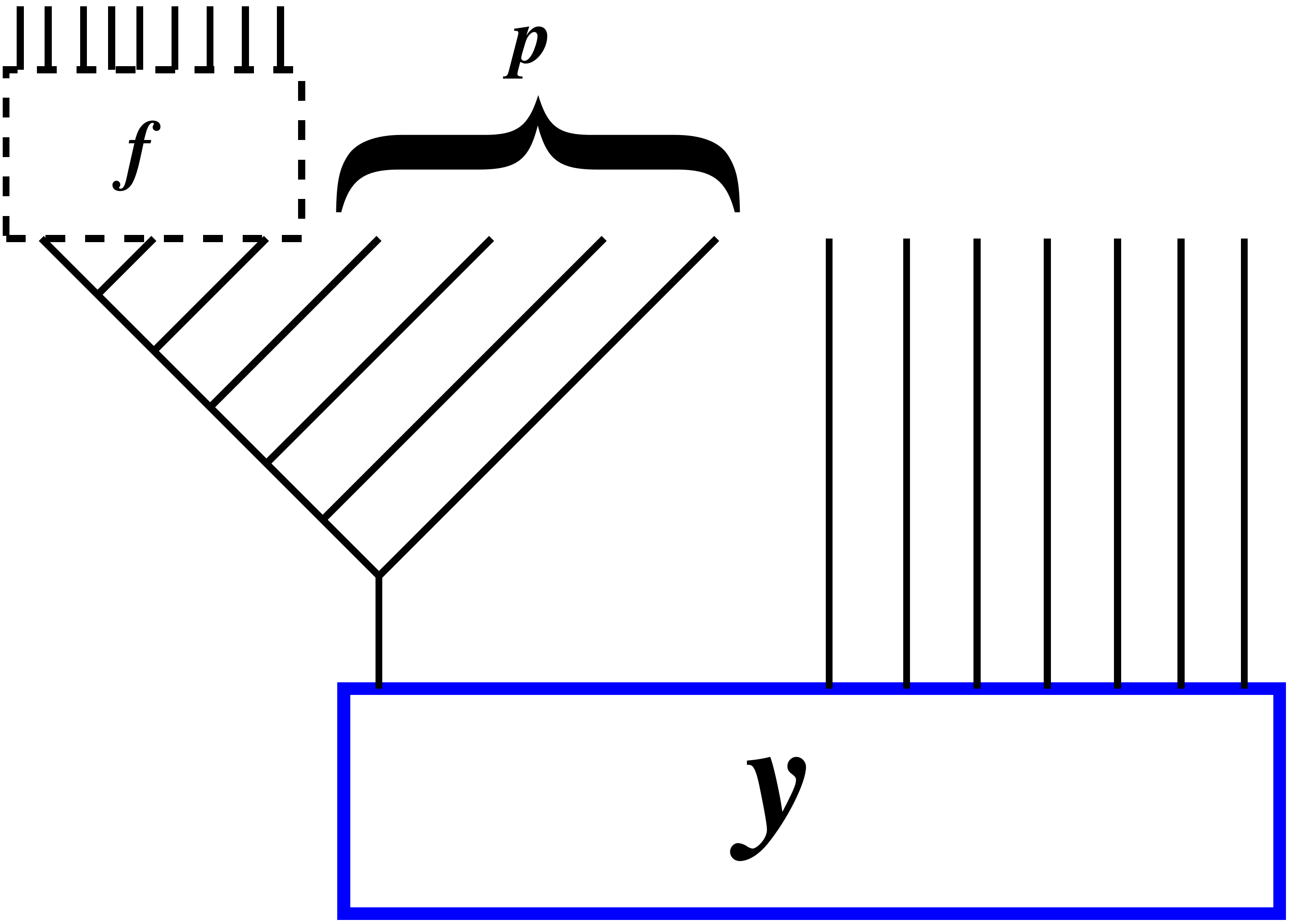} {2in} $$

Observe that $p$ (the number of strings as indicated) can be made arbitrarily large by increasing $n$ whereas the sizes of $f$ and
$g$ depend only on $m$. We see that $\langle \pi_{Y,\phi}(A^n)(\xi),\eta\rangle$ is given by an element of $\mathcal C_0$ that looks like:
$$ \vpic {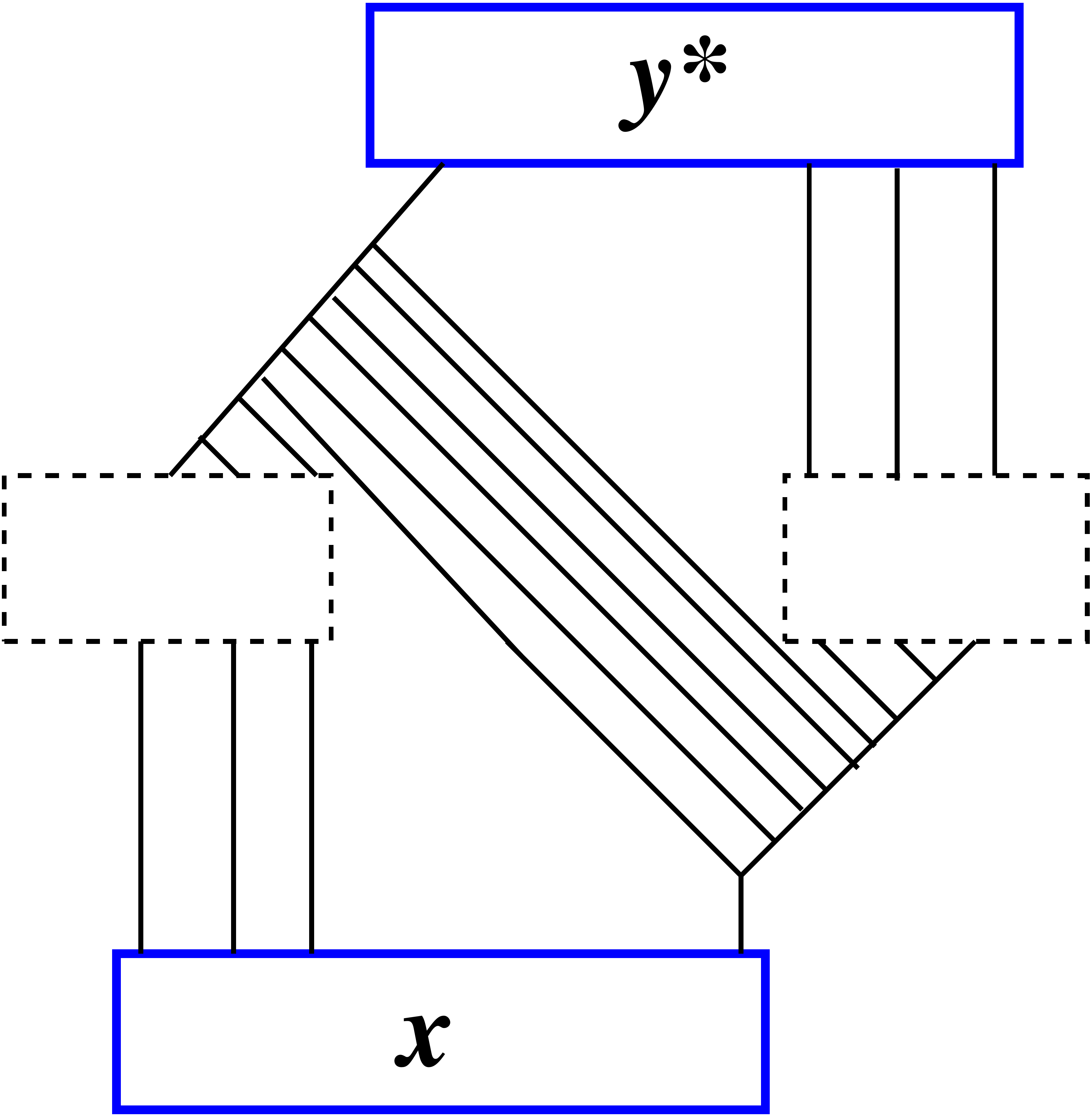} {1.5in} $$
Which can be redrawn in $\mathcal C$ as 
$$ \vpic {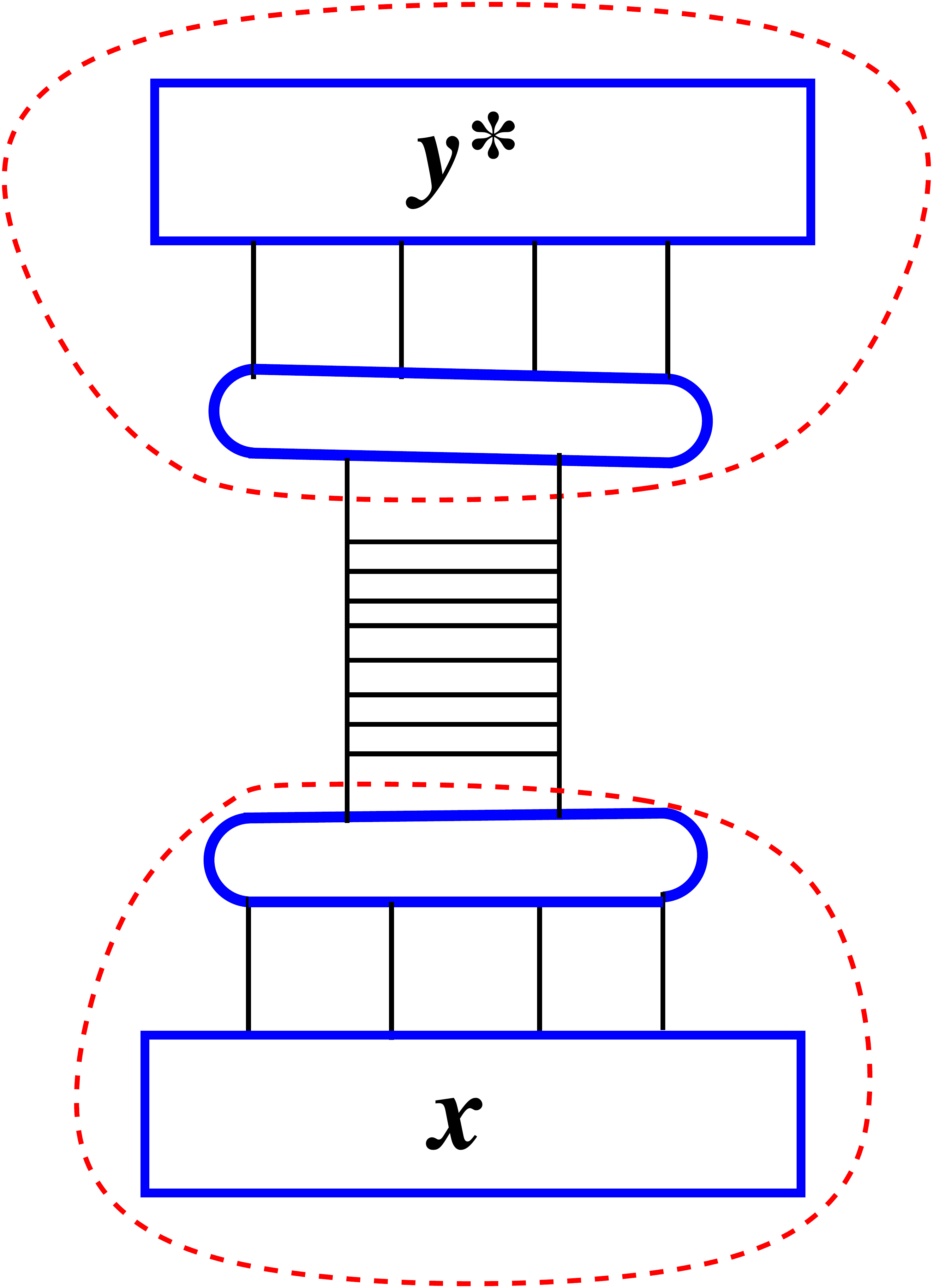} {1.5in} $$
where we observe that, since the dimension of $\mathcal C_2$ is one, the contributions of the top and bottom dotted circles 
are just linear functionals  of $x$ and $y^*$ which do not change as soon as $n$ is large enough. 
By the unitarity property of $Y$, we see that the weak limit of $\pi_{Y,\phi}(A^n)$ exists and has rank 1.

Very little changes for $\pi_{Y,1}(A^n)$. The only difference in the final diagram is a string joining the top of the picture to 
the bottom so that the dotted circles meet three strings instead of two. Thus what is inside them contributes a multiple of
$Y$ and the ladder in the middle introduces a term $t^p$ which tends to zero as $n$, hence $p$ tends to infinity.

\end{proof}

\end{document}